\documentclass[10pt,a4paper,oneside,openany]{book}

\usepackage{latexsym}
\usepackage{amsfonts}
\usepackage{amssymb}
\usepackage{amsmath}
\usepackage{amsthm}
\usepackage{setspace} 
\usepackage{color}
\usepackage{hyperref}
\usepackage{pdfpages}
\usepackage{paralist}    
\usepackage{etoolbox}
\patchcmd{\thebibliography}{\chapter*}{\section*}{}{}

\newcommand{\R}{\mathbb{R}}		

\newcommand{\beq}{\begin{equation}}		
\newcommand{\eeq}{\end{equation}}			
\newcommand{\beqq}{\begin{equation*}}	
\newcommand{\eeqq}{\end{equation*}}		

\newcommand{\id}{1\hspace{-0,9ex}1}

\renewcommand{\P}{\mathbb{P}}

\newcommand{\B}{\mathfrak{B}} 
\newcommand{\F}{\mathcal{F}} 
\newcommand{\A}{\mathcal{A}}

\newcommand{\x}{\text{\scalebox{0.62}{$\mathbb{X}$}}}
\newcommand{\X}{\mathbb{X}}

\allowdisplaybreaks

\newtheorem{theorem}{Theorem}[section]
\newtheorem{lemma}[theorem]{Lemma}
\newtheorem{definition}[theorem]{Definition}
\newtheorem{remark}[theorem]{Remark}

\newtheorem{corollary}[theorem]{Corollary}

\newtheorem{proposition}[theorem]{Proposition}

\newtheorem{stepp}{\noindent\bf{Step}}

{\begin{stepp}\rm}{\rm \end{stepp}}

\begin{document}
\pagestyle{headings} \thispagestyle{headings} \thispagestyle{empty}

\noindent\rule{15.812cm}{0.4pt}
\begin{center}
\textsc{Abstract Cauchy Problems in separable Banach Spaces driven by random Measures: Existence \& Uniqueness }
\end{center}
\begin{center}
	by
\end{center} 
\begin{center}
		\textsc{Alexander Nerlich\footnote{Affiliation: Ulm University}\footnote{Affiliation's address: 89081 Ulm, Helmholtzstr. 18, Germany}\footnote{Author's E-Mail: alexander.nerlich@uni-ulm.de}}
\end{center}
\noindent\rule{15.812cm}{0.4pt}
\vspace{0.4cm}
\pagestyle{myheadings} 
\begin{center} 
	{\large ABSTRACT} 
\end{center}
The purpose of this paper is to study stochastic evolution inclusions of the form
\begin{align*}
 \eta(t,z) N_{\Theta}(dt \otimes z)\in dX(t)+\A X(t)dt,
\end{align*}
where $\A$ is a multi-valued operator acting on a separable Banach space and $N_{\Theta}$ is the counting measure induced by a point process $\Theta$. Firstly, we will set up the concepts of strong and mild solutions; then we will derive existence as well as uniqueness criteria for these kinds of solutions and give a representation formula for the solutions.\\
The results will be formulated by means of nonlinear semigroup theory and except for separability, no assumptions on the underlying Banach space are required.

\section{Introduction}

Nonlinear SPDEs as well as nonlinear PDEs are both vibrant areas of research. Moreover, the theory of nonlinear semigroups and m-accretive operators is a powerful tool to establish the existence of unique solutions for many nonlinear PDEs; including the weighted $p$-Laplacian evolution equation (cf. \cite{mazon}), which serves as a model example throughout this paper.\\ 
Probably one of the most celebrated results regarding nonlinear semigroup theory reads as follows: If $(V,||\cdot||_{V})$ is an arbitrary (real) Banach space and $\A:D(\A)\rightarrow 2^{V}$ is an m-accretive, densely defined operator, then the initial value problem
\begin{align} 
\label{acp}
0 \in u^{\prime}(t)+\A u(t),~\text{a.e. } t \in (0,\infty),~u(0)=v,
\end{align}
has for any $v \in V$ a uniquely determined mild solution, denoted by $T_{\A}(\cdot)v:[0,\infty)\rightarrow V$, see  \cite[Prop. 3.7]{BenilanBook}.\\ 
The purpose of this paper is to derive a similar result for the stochastic evolution inclusion
\begin{align}
\tag{ACPRM}
\label{acprm}
\eta(t,z) N_{\Theta}(dt \otimes z)\in dX(t)+\A X(t)dt.
\end{align}
Surprisingly, it seems that there are very few results connecting abstract Cauchy problems governed by $m$-accretive, multi-valued operators and SPDEs on separable Banach spaces.\\

Before stating our results in more detail, let us give the reader an intuition on how to define what a solution of (\ref{acprm}) is:\\
To this end, let $(\Omega,\F,\P)$ be a complete probability space, $(Z,\mathcal{Z})$ a measurable space and let\linebreak $N_{\Theta }: (\mathfrak{B}((0,\infty))\otimes \mathcal{Z})\times \Omega \rightarrow \mathbb{N}_{0}\cup \{\infty\}$ be the counting measure induced by a finite and simple point process $\Theta$. Consequently, the noise term "$\eta(t,z) N_{\Theta}(d\tau \otimes z)$" is a pure jump noise; in particular we do not assume that the random measure $N_{\Theta}$ is compensated. In addition, let $(V,||\cdot||_{V})$ be a real, separable Banach space and let $\eta: (0,\infty)\times Z \times \Omega \rightarrow V$ be jointly measurable. Then for a multi-valued operator $\A:D(\A)\rightarrow 2^{V}$, one would at first try to define a solution of (\ref{acprm}) as a process $X:[0,\infty)\times \Omega \rightarrow V$ which is sufficiently regular and fulfills
\begin{align*}
\int \limits_{(0,t]\times Z}\eta(\tau,z)N_{\Theta}(d\tau\otimes z) \in X(t)-x+\int \limits_{0}\limits^{t}\A X(\tau)d\tau,
\end{align*}
where $x:\Omega \rightarrow V$ is an initial, i.e. $X(0)=x$. The obvious issue is that $\A$ takes values in the power set of $V$. Consequently, one either has to somehow define the set-valued integral, or one has to "pick" for each $\tau$ and $\omega$ an element of $\A X(\tau,\omega)$ by some rule. We choose to do the latter. To define this rule, assume that $\A$ is m-accretive and densely defined and let $T_{\A}$ denote the semigroup associated to $\A$. Moreover, assume that $\A$ admits an infinitesimal generator $\A^{\circ}:V\rightarrow V$, that is
\begin{align} 
-\lim \limits_{h \searrow 0} \frac{T_{\mathcal{A}}(h)v-v}{h} =: \mathcal{A}^{\circ}v \in \A v,
\end{align}
for all $ v \in D(\mathcal{A})$ and $\A^{\circ}v=0$ for all $v \in V \setminus D(\A)$. (In the nonlinear case, the existence of that limit is an assumption and not necessarily fulfilled.)\\
Consequently, we have found a rule and would like to define a solution as a process fulfilling
\begin{align*}
\int \limits_{(0,t]\times Z}\eta(\tau,z)N_{\Theta}(d\tau\otimes z) = X(t)-x+\int \limits_{0}\limits^{t}\A^{\circ} X(\tau)d\tau.
\end{align*}
The issue with this equation is that one needs $\A^{\circ}X \in L^{1}((0,t);V)$ for all $t>0$ with probability one. To get an existence result as applicable as possible, we will therefore formulate the preceding equation in a weak sense; more precisely, we will term strong solution, as a process $X$ fulfilling
\begin{align*}
\int \limits_{(0,t]\times Z}\langle\Psi,\eta(\tau,z)\rangle_{V}N_{\Theta}(d\tau\otimes z)=\langle\Psi,X(t)-x\rangle_{V}+\int \limits_{0}\limits^{t}\langle\Psi,\A^{\circ}X(\tau)\rangle_{V}d\tau,
\end{align*}
for all $\psi \in V^{\ast}$, where $V^{\prime}$ denotes the dual of $V$, $\langle\cdot,\cdot\rangle_{V}$ the duality between $V$ and $V^{\prime}$ and $V^{\ast}\subseteq V^{\prime}$ is a set which separates points. Of course, the process $X$ also has to fulfill some regularity assumptions, which mainly serve to make sure the uniqueness of solutions.\\
In addition, we will introduce a "mild solution of (\ref{acprm})", as a process which can be approximated in some sense by strong solutions.\\

Having done so, we shall see that (\ref{acprm}) has for any $\F$-$\B(V)$-measurable initial $x:\Omega \rightarrow V$ a unique mild solution, if: $\A$ is densely defined, m-accretive, domain invariant, admits an infinitesimal generator and if there is a dense subset $\mathcal{V}\subseteq V$, which is invariant w.r.t. $T_{\A}$ and such that $\langle\Psi, \A^{\circ}T_{\A}(\cdot)v\rangle_{V}\in L^{1}(0,t)$ for all $t>0$, $v \in \mathcal{V}$ and $\Psi \in V^{\ast}$. Particularly, this result only requires that $\eta$ and $x$ are measurable. Moreover, it will be demonstrated that mild solutions depend Lipschitz continuously on the initial $x$ and the drift $\eta$. Furthermore, if $x \in \mathcal{V}$ and $\eta(t,z) \in \mathcal{V}$ for all $t>0$, $z \in Z$ almost surely, then the mild solution is even a strong one. Along the way, a representation formula for mild (and therefore also for strong) solutions is derived.\\
In addition, these results will be exemplified by the weighted $p$-Laplacian evolution Equation with Neumann boundary conditions acting on an $L^{1}$-space.\\

The main advantages of employing the theory of $m$-accretive operators to solve (\ref{acprm}) is that this works on any separable Banach space. Moreover, the fairly lean assumptions on $\A$ allow to consider a large group of operators, such as the weighted $p$-Laplacian operator with a weight function only fulfilling boundedness, measurability and integrability assumptions, but no differentiability assumptions.\\ 

The investigation of (\ref{acprm}) will be continued in \cite{ich2}. There, we employ the representation formula derived here, to prove intriguing asymptotic results for the solutions, such as the strong law of large numbers and the central limit theorem.\\ 

That all of this works is highly owed to the fact that the noise term is a pure jump noise. However, it seems reasonable to conjecture that one can extend these results to compensated random measures by applying the theory of $m$-accretive operators for inhomogeneous Cauchy problems.\\

This paper is structured as follows: Section \ref{basicdef} is a collection of preliminary results and clarifies the notation we use. Section \ref{section_exun} is this paper's centerpiece; all general results regarding existence and uniqueness are proven there. And last but not least, the applicability of these results to the weighted $p$-Laplacian evolution equation is demonstrated in Section \ref{section_plaplace}.\\

\section{Preliminary Results and Notation}
\label{basicdef}
Throughout everything which follows $(\Omega,\F,\P)$ denotes a complete probability space and $(Z,\mathcal{Z})$ a measurable space.\\ 
Let us start with some stochastic preliminaries and proceed with the functional analytic ones.\\
We call a mapping $\theta:D(\theta)\rightarrow Z$, where $D(\theta)\subseteq (0,\infty)$ is countable, a point function. Moreover, $\pi(Z)$ denotes the set of all point functions mapping into $Z$ and we equip this space with the $\sigma$-algebra
\begin{align*}
\Pi(Z):=\sigma \big(\big\{\theta\in \pi(Z): \#\{t \in D(\theta): (t,\theta(t))\in U\}=k\big\};~k \in \mathbb{N}_{0},~U \in \mathfrak{B}((0,\infty))\otimes \mathcal{Z}\big),
\end{align*}
where $\B(T)$ always denotes the Borel $\sigma$-algebra, if $(T,\mathcal{T})$ is a topological space.\\
In addition, a mapping $\Theta :\Omega \rightarrow \pi(Z)$ which is $\F-\Pi(Z)$-measurable, is called a random point function, or point process. Moreover, for a point process $\Theta :\Omega \rightarrow \pi(Z)$, we introduce the mapping\linebreak $N_{\Theta }: (\mathfrak{B}((0,\infty))\otimes \mathcal{Z})\times \Omega \rightarrow \mathbb{N}_{0}\cup \{\infty\}$ by
\begin{align*}
N_{\Theta }(U,\omega):=\#\{t \in D(\Theta (\omega)): (t,\Theta (\omega)(t))\in U\},~\forall U \in \mathfrak{B}((0,\infty))\otimes \mathcal{Z},~ \omega \in \Omega
\end{align*}
and refer to it as the counting measure induced by $\Theta $.\\
It is plain to verify that the mapping $\mathfrak{B}((0,\infty))\otimes \mathcal{Z} \ni U \mapsto N_{\Theta }(U,\omega)$ is a measure for each $\omega \in \Omega$ and that $\Omega \ni \omega \mapsto N_{\Theta }(U,\omega)$ is a (extended) real-valued random variable for each $\omega \in \Omega$. (Hereby extended refers to the fact that this random variable might take the value $+\infty$.)\\
Note that, by definition, any point process $\Theta$ is simple, i.e. $N_{\Theta}(\{t\times z\},\omega)\leq 1$ for all $(t,z)\in (0,\infty)\times Z$ and $\omega \in \Omega$.\\
A point process $\Theta :\Omega \rightarrow \pi(Z)$, or the random measure $N_{\Theta }$ induced by $\Theta $, is called finite if \linebreak$\mathbb{E} N_{\Theta }((0,t]\times Z)<\infty$ for all $\forall t>0$. It is easy to infer that this implies $N_{\Theta }((0,t]\times Z)<\infty$ for all $t \in (0,\infty)$ with probability one. 

\begin{remark}\label{remark_hitting} Let $N_{\Theta}$ be the counting measure induced by a finite point process $\Theta:\Omega \rightarrow \Pi(Z)$. Then there is a $\P$-null-set $M \in \F$, such that $N_{\Theta }((0,t]\times Z,\omega)<\infty$ for all $t>0$ and $\omega \in \Omega \setminus M$. Hence, $D(\Theta (\omega))\cap (0,t]$ contains only finitely many elements for any $t>0$; which yields that $D(\Theta (\omega))$ is an isolated set for any $\omega \in \Omega \setminus M$. Therefore, we can find a sequence of mappings $(\alpha_{m})_{m \in \mathbb{N}}$, with $\alpha_{m}:\Omega \rightarrow (0,\infty)$, such that
	\begin{enumerate}
		\item $D(\Theta (\omega))=\{\alpha_{1}(\omega),\alpha_{2}(\omega),...\}$ for all $\omega \in \Omega \setminus M$ and
		\item $0 < \alpha_{m}(\omega)<\alpha_{m+1}(\omega)<\infty$ for all $m \in \mathbb{N}$ and $\omega \in \Omega \setminus M$.
	\end{enumerate}
	The sequence of mappings $(\alpha_{m})_{m \in \mathbb{N}}$ fulfilling these two assertions is obviously unique on $\Omega \setminus M$. 	We will refer to the (up to a $\P$-null-set) uniquely determined sequence fulfilling the assertions i)-ii), as the sequence of hitting times induced by $\Theta $.\\
	One instantly verifies that each $\alpha_{m}$ is $\F$-$\B((0,\infty))$-measurable and that $\lim \limits_{m \rightarrow \infty}\alpha_{m}=\infty$ almost surely. Moreover, with slightly more effort one verifies that the mapping defined by $\Omega \ni \omega \mapsto \Theta (\omega)(\alpha_{m}(\omega))$ is $\F-\mathcal{Z}$-measurable. 
\end{remark}

For a function $f:(0,\infty)\times Z \times \Omega \rightarrow \mathbb{R}$ which is $\mathfrak{B}((0,\infty))\otimes \mathcal{Z}\otimes \F-\mathfrak{B}(\mathbb{R})$-measurable and a finite point measure $N_{\Theta}$, we introduce
\begin{align}
\label{eq_integral}
\left( \int \limits_{(0,t]\times Z}f(\tau,z)N_{\Theta }(d\tau \otimes dz)\right)(\omega):=\int \limits_{(0,t]\times Z}f(\tau,z,\omega)N_{\Theta }(d\tau \otimes dz,\omega),~\forall t>0,~\P-\text{a.e. } \omega \in \Omega.
\end{align}
Hereby the right hand side is understood as a Lebesgue integral w.r.t. the measure $N(\cdot,\omega)$. Basic properties of this integral will be stated at this section's end, see Lemma \ref{lemma_intex}. Particularly, it is proven there that this integral is indeed finite for any measurable $f$ and finite point measure $N_{\Theta}$. 

\begin{remark} Throughout everything which follows, $\Theta:\Omega \rightarrow \pi(Z)$ denotes a finite point process and $N_{\Theta }: (\mathfrak{B}((0,\infty))\otimes \mathcal{Z})\times \Omega \rightarrow \mathbb{N}_{0}\cup \{\infty\}$ denotes the counting measure induced by $\Theta$. Moreover $(\alpha_{m})_{m \in \mathbb{N}}$ denotes the sequence of hitting times induced by $\Theta$. Finally, for notational convenience we also introduce $\alpha_{0}:\Omega \rightarrow \mathbb{R}$, with $\alpha_{0}:=0$.
\end{remark}

Now we will proceed with the functional analytic preliminaries: Let $(V,||\cdot||_{V})$ denote a real, separable Banach space with dual space $V^{\prime}$. Moreover, let $\langle\cdot,\cdot\rangle_{V}$ denote the duality between $V$ and $V^{\prime}$. As usually, a subset $V^{\ast}\subseteq V^{\prime}$ is said to separate points, if for all $v \in V$ we have that $\langle\Phi,v\rangle_{V}=0$ for all $\Phi \in V^{\ast}$ implies $v=0$.\\
In addition, let
\begin{align*}
W^{1,1}([a,b];V):=\{f:[a,b]\rightarrow V:~f\text{ is absolutely continuous and differentiable a.e.}\}.
\end{align*}
Moreover, for a measurable space $(K,\Sigma)$, we denote by $\mathcal{M}(K,\Sigma;V)$ the space of all functions $f:K \rightarrow V$ which are $\Sigma-\B(V)$-measurable; and we may simple write $\mathcal{M}(K;V)$, if it is clear which $\sigma$-algebra is meant; particularly:
\begin{align*}
\mathcal{M}((0,\infty)\times Z \times \Omega;V):=\mathcal{M}((0,\infty)\times Z \times \Omega, \B((0,\infty)\otimes \mathcal{Z} \otimes \F;V) \text{ and } \mathcal{M}(\Omega;V):=\mathcal{M}(\Omega,\F;V).
\end{align*}
Now we also need to briefly recall some definitions and results regarding nonlinear semigroup theory. The reader is referred to \cite{BenilanBook} for a comprehensive introduction to this topic. Moreover, \cite{acmbook} deals with existence, uniqueness and asymptotic results for many initial valued problems and this book's appendix contains a more concise introduction to nonlinear semigroups.\\
Let $\mathcal{A}:V\rightarrow 2^{V}$ be a multi-valued operator, then we introduce $D(\mathcal{A}):=\{v\in V: \mathcal{A}v\neq \emptyset\}$ and may write $\mathcal{A}:D(\A)\rightarrow 2^{V}$. Moreover, we call this operator single-valued if $\mathcal{A}v$ contains precisely one element for all $v\in D(\mathcal{A})$. In addition, $G(\A):=\{ (v,\hat{v}): \hat{v}\in \A v \}$ is the graph of $v$. We identify operators with its graph, and may simply write $(v,\hat{v})\in \A$, instead of $v \in D(\A)$ and $\hat{v} \in \A v$.\\
In addition, $\mathcal{A}:D(\mathcal{A})\rightarrow 2^{V}$ is called accretive, if $||v_{1}-v_{2}||_{V}\leq ||v_{1}-v_{2}+\alpha(\hat{v}_{1}-\hat{v}_{2})||_{V}$ for all $\alpha>0$ and $(v_{1},\hat{v}_{1}),~(v_{2},\hat{v}_{2})\in \A$; m-accretive, if it is accretive and $R(Id+\alpha\mathcal{A})=V$ for all $\alpha>0$; and densely defined, if $\overline{D(\mathcal{A})}=V$.\\
Using these simple definitions enables us to invoke the following well-known result:

\begin{remark}\label{remark_ms} Let $\mathcal{A}:D(\mathcal{A})\rightarrow 2^{V}$ be m-accretive and densely defined; moreover, let $v\in V$. Then the initial initial value problem
	\begin{align}
	\label{remark_mseq}
	0 \in u^{\prime}(t)+\mathcal{A}u(t),~\text{for a.e. }t\in (0,\infty),~u(0)=v,
	\end{align}
	has precisely one mild solution. The reader is referred to \cite[Proposition 3.7]{BenilanBook} for a proof and to \cite[Definition 1.3]{BenilanBook} for the definition of mild solution.\\
	For a given m-accretive and densely defined operator $\mathcal{A}:D(\mathcal{A})\rightarrow 2^{V}$, we denote for each $v \in V$ by $T_{\mathcal{A}}(\cdot)v:[0,\infty)\rightarrow V$ the uniquely determined mild solution of (\ref{remark_mseq}). The most important properties of $T_{\A}$ which are needed throughout this paper are as follows:
	\begin{enumerate}
		\item Contractivity: $||T_{\mathcal{A}}(t)v_{1}-T_{\mathcal{A}}(t)v_{2}||_{V} \leq ||v_{1}-v_{2}||_{V}$, for all $t\in [0,\infty)$ and $v_{1},~v_{2}\in V$. (cf. \cite[Theorem 3.10.i)]{BenilanBook})
		\item Continuity: $[0,\infty) \times V \ni (t,v) \mapsto T_{\mathcal{A}}(t)v$ is jointly continuous. (cf. \cite[Theorem 3.10.iii)]{BenilanBook}
		\item Semigroup Property: $T_{\mathcal{A}}(t_{1}+t_{2})v=T_{\mathcal{A}}(t_{2})T_{\mathcal{A}}(t_{1})v$ and $T_{\mathcal{A}}(0)v=v$ for all $t_{1},~t_{2}\in[0,\infty)$ and $v \in V$. (cf. \cite[Theorem 1.10]{BenilanBook})
	\end{enumerate}
The family of mappings $(T_{\A}(t))_{t \geq 0}$ will be called the semigroup associated to $\A$.
\end{remark}

\begin{definition} Let $\mathcal{A}:D(\mathcal{A})\rightarrow 2^{V}$ be m-accretive and densely defined. Moreover, let $\tilde{V}\subseteq V$. Then we say that $\tilde{V}$ is an invariant set w.r.t. $T_{\A}$, if $T_{\A}(t)\tilde{v} \in \tilde{V}$ for all $t \in [0,\infty)$ and $\tilde{v}\in\tilde{V}$.\\
	Moreover, $T_{\A}$ is called domain invariant, if $T_{\A}(t)v \in D(\A)$ for all $t \in (0,\infty)$ and $v \in V$.
\end{definition}

\begin{definition} Let $\mathcal{A}:D(\mathcal{A})\rightarrow 2^{V}$ be m-accretive and densely defined. Then we say that $T_{\A}$ admits an infinitesimal generator, if there is an operator $\mathcal{A}^{\circ}:V\rightarrow V$ such that
	\begin{align}
	\label{def_geneq}
	-\lim \limits_{h \searrow 0} \frac{T_{\mathcal{A}}(h)v-v}{h} = \mathcal{A}^{\circ}v \in \A v,
	\end{align}
	for all $ v \in D(\mathcal{A})$ and $\A^{\circ}v=0$ for all $v \in V \setminus D(\A)$. In this case, we call $\mathcal{A}^{\circ}$ the infinitesimal generator of $T_{\A}$.
\end{definition}

Using these definitions enables us to rigorously define the different notions of solutions of (\ref{acprm}).

\begin{definition}\label{def_solutionacprm} Let $(V,||\cdot||_{V})$ denote a separable Banach space, let $\eta \in \mathcal{M}((0,\infty)\times Z \times \Omega; V)$ and $x \in \mathcal{M}(\Omega;V)$. In addition, let $V^{\ast}\subseteq V^{\prime}$ be a set that separates points and let $\A:D(\A)\rightarrow 2^{V}$ be a densely defined, m-accretive operator which admits an infinitesimal generator $\A^{\circ}:V\rightarrow V$.\\
Then a $\B([0,\infty))\otimes \F$-$\B(V)$-measurable stochastic process $X:[0,\infty)\times \Omega \rightarrow V$ is called a strong solution of (\ref{acprm})$\{x,\eta,V^{\ast}\}$ if all of the following assertions hold for $\P$-a.e. $\omega \in \Omega$.
\begin{enumerate}
	\item $X(0,\omega)=x(\omega)$,
	\item the mapping $[0,\infty)\ni t \mapsto X(t,\omega)$ is c\`{a}dl\`{a}g,
	\item $X(t,\omega)\in D(\mathcal{A}),~\forall t \in (0,\infty)\setminus\{\alpha_{m}(\omega):m\in \mathbb{N}\}$,
	\item $\forall m \in \mathbb{N}_{0},~\forall [a,b]\subseteq (\alpha_{m}(\omega),\alpha_{m+1}(\omega)):~ X(\cdot,\omega)|_{[a,b]} \in W^{1,1}([a,b];V)$,
	\item $\langle \Psi,\A^{\circ}X(\cdot,\omega)\rangle_{V}\in L^{1}(0,t),~\forall t>0,~\Psi \in V^{\ast}$ and
	\item $\langle\Psi,X(t,\omega)-x(\omega)\rangle_{V}+\int \limits_{0}\limits^{t}\langle\Psi,\A^{\circ}X(\tau,\omega)\rangle_{V}d\tau=\int \limits_{(0,t]\times Z}\langle\Psi,\eta(\tau,z,\omega)\rangle_{V}N_{\Theta}(d\tau\otimes z,\omega),~\forall t>0,~\Psi \in V^{\ast}$.
\end{enumerate}
In addition, a $\B([0,\infty))\otimes \F$-$\B(V)$-measurable stochastic process $Y:[0,\infty)\times \Omega \rightarrow V$ is called a mild solution of (\ref{acprm})$\{x,\eta,V^{\ast}\}$, if it fulfills conditions i-iv) with probability one and if there are sequences $(x_{m})_{m\in \mathbb{N}}$, $(\eta_{m})_{m \in \mathbb{N}}$ and $(X_{m})_{m\in\mathbb{N}}$ such that
\begin{enumerate}
	\setcounter{enumi}{6}
	\item $x_{m}\in \mathcal{M}(\Omega;V)$ and $\eta_{m} \in \mathcal{M}((0,\infty)\times Z \times \Omega; V)$ for all $m \in \mathbb{N}$,
	\item $X_{m}:\Omega \times [0,\infty)\rightarrow V$ is a strong solution of (\ref{acprm})$\{x_{m},\eta_{m},V^{\ast}\}$ for all $m \in \mathbb{N}$,
	\item $\lim \limits_{m \rightarrow\infty} \sup \limits_{\tau \in [0,t]} ||X_{m}(\tau)-Y(\tau)||_{V}=0$ for all $t> 0$ almost surely and
	\item $\lim \limits_{m \rightarrow\infty}\int \limits_{(0,t]\times Z}||\eta_{m}(\tau,z)-\eta(\tau,z)||_{V}N_{\Theta}(d\tau\otimes z)=0$ for all $t> 0$ almost surely.
\end{enumerate} 
\end{definition}

\noindent As promised, this section now concludes with stating some basic properties of the integral defined in (\ref{eq_integral}).

\begin{lemma}\label{lemma_intex} Let $M \in \F$ be a $\P$-null-set such that
	\begin{align}
	\label{lemma_intexeq0}
	D(\Theta(\omega))=\{\alpha_{1}(\omega),~\alpha_{2}(\omega),...\},~0<\alpha_{m}(\omega)<\alpha_{m+1}(\omega),~\forall m \in \mathbb{N}\text{ and } \lim \limits_{m \rightarrow \infty}\alpha_{m}(\omega)=\infty,
	\end{align}
	for all $\omega \in \Omega \setminus M$. Moreover, introduce $f\in \mathcal{M}((0,\infty)\times Z \times \Omega;\mathbb{R})$.\\ 
	Then the mapping defined by $\Omega \ni \omega \mapsto f(\alpha_{m}(\omega),\Theta (\omega)(\alpha_{m}(\omega)),\omega):=f_{m}(\omega)$ is well defined on $\Omega \setminus M$ and $\F-\B(\R)$-measurable. In addition, the Lebesgue integral 
	\begin{align}
	\label{lemma_intexeq1}
	\int \limits_{(0,t]\times Z} f(\tau,z,\omega) N(d\tau \otimes z,\omega)
	\end{align}
	exists and is finite for all $t>0$ and $\omega \in \Omega \setminus M$. Moreover, the mapping defined by\linebreak $\Omega \ni \omega \mapsto \int \limits_{(0,t]\times Z} f(\tau,z,\omega) N(d\tau \otimes z,\omega)$ is well-defined on $\Omega \setminus M$, and it is $\F-\B(\R)$-measurable for all $t>0$. Finally, the assertion
	\begin{align}
	\label{lemma_intexeq2}
	\int \limits_{(0,t]\times Z} f(\tau,z,\omega) N(d\tau \otimes z,\omega) = \sum \limits_{m =1}\limits^{\infty} \sum \limits_{k =1 }\limits^{m}f_{k}(\omega)\id_{[\alpha_{m}(\omega),\alpha_{m+1}(\omega))}(t)
	\end{align}
	is valid for all $t>0$ and $\omega \in \Omega \setminus M$.
\end{lemma}
\begin{proof} Employing Remark \ref{remark_hitting} yields that each $f_{m}$ is the composition of measurable functions and consequently $\F-\mathfrak{B}(\R)$-measurable.\\ 
	Now note that it is plain that the mapping defined by $(0,t] \times Z \ni (\tau,z)\mapsto f(\tau,z,\omega)$ is $\mathfrak{B}((0,t])\otimes \mathcal{Z}-\B(\R)$ measurable for all $\omega \in \Omega$ and $t>0$. Consequently, it follows that the Lebesgue integral considered in (\ref{lemma_intexeq1}) is well defined and finite, if
	\begin{align}
	\label{lemma_intexproof1}
	\int \limits_{(0,t]\times Z} |f(\tau,z,\omega)| N(d\tau \otimes z,\omega)<\infty,~\forall t>0,~\omega \in \Omega \setminus M.
	\end{align}
	To this end, note that
	\begin{align}
	\label{lemma_intexproof2}
	N_{\Theta }\Big(\big(\alpha_{m}(\omega),\alpha_{m+1}(\omega)\big)\times Z,\omega\Big)=0,~\forall m \in \mathbb{N}_{0},~\omega \in \Omega \setminus M
	\end{align}
	as well as
	\begin{align}
	\label{lemma_intexproof3}
	N_{\Theta }\big(\{\alpha_{m}(\omega)\}\times Z,\omega\big)=N_{\Theta }\big(\{\alpha_{m}(\omega)\}\times \{\Theta (\omega)(\alpha_{m}(\omega))\},\omega\big)=1,~\forall m \in \mathbb{N},~\omega \in \Omega \setminus M.
	\end{align}
	Moreover, for a given $t>0$ and $\omega \in \Omega \setminus M$ there is an $m \in \mathbb{N}$, such that $t< \alpha_{k}(\omega)$ for all $k \in \mathbb{N}\setminus\{1,...,m\}$. This combined with the preceding two equalities clearly yields (\ref{lemma_intexproof1}).\\ 
	Moreover, note that the right-hand-side of (\ref{lemma_intexeq2}) defines an $\F$-$\B(R)$-measurable mapping. Consequently, as $\F$ is complete, the claim follows as soon as (\ref{lemma_intexeq2}) is proven. This is easily deduced from (\ref{lemma_intexproof2}) and (\ref{lemma_intexproof3}), since these two equations yield
	\begin{enumerate}
		\item $	\int \limits_{(0,\alpha_{m}(\omega)]\times Z}f(\tau,z,\omega)N_{\Theta }(d\tau \otimes dz,\omega)= \sum \limits_{k =1 }\limits^{m} f_{k}(\omega)$ for all $m \in \mathbb{N}$,
		\item $\int \limits_{(0,t]\times Z}f(\tau,z,\omega)N_{\Theta }(d\tau \otimes dz,\omega)\id_{[0,\alpha_{1}(\omega))}(t)=0$ for all $t>0$ and
		\item $\int \limits_{(0,t]\times Z}f(\tau,z,\omega)N_{\Theta }(d\tau \otimes dz,\omega)\id_{[\alpha_{m}(\omega),\alpha_{m+1}(\omega))}(t)=\sum \limits_{k =1 }\limits^{m} f_{k}(\omega)\id_{[\alpha_{m}(\omega),\alpha_{m+1}(\omega))}(t)$, for all $m \in \mathbb{N}$ and $t>0$,
	\end{enumerate}
	for all $\omega \in \Omega \setminus M$.
\end{proof}

Similar versions of the preceding result can be found in the literature. For example a similar result (for the case that $N_{\Theta }$ is a Poisson random measure) can be found in \cite[Corollary 3.4]{knoche}, nevertheless we were unable to find it stated precisely as above anywhere in the literature.

\section{Abstract Cauchy problems in separable Banach spaces driven by random Measures: Existence and Uniqueness}
\label{section_exun}
Now we will turn to the main objective of this paper, namely: When is there a unique (mild or strong) solution of (\ref{acprm}). \\

Throughout this section $(V,||\cdot||_{V})$ denotes a real, separable Banach space with dual space $V^{\prime}$. In addition,  $\A:D(\A)\rightarrow 2^{V}$ is a densely defined, domain invariant, m-accretive operator which admits an infinitesimal generator $\A^{\circ}:V\rightarrow V$. Finally $(T_{\A}(t))_{t \geq 0}$ denotes the semigroup associated to $\A$, see Remark \ref{remark_ms}.\\ 
At first we are going to tackle the problem of uniqueness of solutions of (\ref{acprm}), afterwards we will derive the existence.

\begin{lemma}\label{lemmma_solequivalence} Let $V^{\ast} \subseteq V^{\prime}$ be a set that separates points, $\eta \in \mathcal{M}((0,\infty)\times Z \times \Omega;V)$, $x\in \mathcal{M}(\Omega;V)$ and introduce $\eta_{k}(\omega):=\eta(\alpha_{k}(\omega),\Theta(\omega)(\alpha_{k}(\omega)),\omega)$ for all $k \in \mathbb{N}$ and $\P$-a.e. $\omega \in \Omega$. Then $\eta_{k} \in \mathcal{M}(\Omega;V)$ for all $k \in \mathbb{N}$ and a $\B([0,\infty))\otimes \F$-$\B(V)$-measurable stochastic process $X:[0,\infty)\times\Omega\rightarrow V$ is a strong solution of (\ref{acprm})$\{x,\eta,V^{\ast}\}$ if and only if it fulfills \ref{def_solutionacprm}.i-v) and 
\begin{align}
\label{lemma_solrepreq}
	 \langle\Psi,X(t)-x\rangle_{V}+\int \limits_{0}\limits^{t}\langle\Psi,\A^{\circ}X(\tau)\rangle_{V}d\tau=\sum \limits_{m =1}\limits^{\infty} \sum \limits_{k =1 }\limits^{m}\langle\Psi,\eta_{k}\rangle_{V}\id_{[\alpha_{m},\alpha_{m+1})}(t),~\forall t>0,~\Psi \in V^{\ast}.
\end{align}
almost surely.
\end{lemma}
\begin{proof} Firstly, appealing to Remark \ref{remark_hitting} yields that each $\eta_{k}$ is, up to a $\P$-null-set, well-defined and that $\eta_{k}$ is the composition of measurable functions and consequently $\F$-$\B(V)$-measurable.\\
Lemma \ref{lemma_intex} yields that there is a $\P$-null-set $M \in \F$ such that for all $\Psi \in V^{\ast}$, we have
\begin{align*}
	\int \limits_{(0,t]\times Z} \langle\Psi,\eta(\tau,z,\omega)\rangle_{V} N(d\tau \otimes z,\omega) = \sum \limits_{m =1}\limits^{\infty} \sum \limits_{k =1 }\limits^{m}\langle\Psi,\eta_{k}(\omega)\rangle_{V}\id_{[\alpha_{m}(\omega),\alpha_{m+1}(\omega))},~\forall t>0,~\omega \in \Omega \setminus M
\end{align*}
Consequently, we get that \ref{def_solutionacprm}.vi) holds almost surely if and only if (\ref{lemma_solrepreq}) does, which concludes the proof. 
\end{proof}

\begin{proposition}\label{prop_unique} Let $V^{\ast} \subseteq V^{\prime}$ be a set that separates points, $\eta_{1},~\eta_{2}\in \mathcal{M}((0,\infty)\times Z \times \Omega;V)$ and $x_{1},~x_{2}\in \mathcal{M}(\Omega;V)$. Moreover, assume $X_{i}:[0,\infty)\times \Omega\rightarrow V$ is a strong solution of (\ref{acprm})$\{x_{i},\eta_{i},V^{\ast}\}$ for $i=1,2$. Then we have
\begin{align}
\label{prop_uniqueeq}
||X_{1}(t)-X_{2}(t)||_{V}\leq ||x_{1}-x_{2}||_{V}+\int \limits_{(0,t]\times Z} ||\eta_{1}(\tau,z)-\eta_{2}(\tau,z)||_{V} N_{\Theta}(d\tau\otimes z),~\forall t \geq 0,
\end{align}
almost surely. 
\end{proposition}
\begin{proof} Firstly, by Lemma \ref{lemmma_solequivalence} and Remark \ref{remark_hitting} we get that there is a $\P$-null-set $M  \in \F$ such that
\begin{align}
\label{prop_uniqueproofeq0}
D(\Theta(\omega))=\{\alpha_{1}(\omega),~\alpha_{2}(\omega),...\},~0<\alpha_{m}(\omega)<\alpha_{m+1}(\omega),~\forall m \in \mathbb{N}\text{ and } \lim \limits_{m \rightarrow \infty}\alpha_{m}(\omega)=\infty,
\end{align}
and
\begin{enumerate}
	\item $X_{i}(0,\omega)=x_{i}(\omega)$,
	\item the mapping $[0,\infty)\ni t \mapsto X_{i}(t,\omega)$ is c\`{a}dl\`{a}g, 
	\item $X_{i}(t,\omega)\in D(\mathcal{A}),~\forall t \in (0,\infty)\setminus\{\alpha_{m}(\omega):m\in \mathbb{N}\}$,
	\item $\forall m \in \mathbb{N}_{0},~\forall [a,b]\subseteq (\alpha_{m}(\omega),\alpha_{m+1}(\omega)):~ X_{i}(\cdot,\omega)|_{[a,b]} \in W^{1,1}([a,b];V)$, 
	\item $\langle \Psi,\A^{\circ}X_{i}(\cdot,\omega)\rangle_{V}\in L^{1}(0,t),~\forall t>0,~\Psi \in V^{\ast}$ and
	\item $\langle\Psi,X_{i}(t,\omega)-x_{i}(\omega)\rangle_{V}+\int \limits_{0}\limits^{t}\langle\Psi,\A^{\circ}X_{i}(\tau,\omega)\rangle_{V}d\tau=\sum \limits_{m =1}\limits^{\infty} \sum \limits_{k =1 }\limits^{m}\langle\Psi,\eta_{i,k}(\omega)\rangle_{V}\id_{[\alpha_{m}(\omega),\alpha_{m+1}(\omega))}(t),~\forall t>0$,\linebreak$\Psi \in V^{\ast}$, where $\eta_{i,k}(\omega):=\eta_{i}(\alpha_{k}(\omega),\Theta(\omega)(\alpha_{k}(\omega)),\omega )$ for all $k \in \mathbb{N}$,
\end{enumerate}
for all $\omega \in \Omega \setminus M$ and $i=1,2$.\\
Moreover, Lemma \ref{lemma_intex} yields that it suffices to prove that
\begin{align}
\label{prop_uniqueproofeq1}
||X_{1}(t,\omega)-X_{2}(t,\omega)||_{V}\leq ||x_{1}(\omega)-x_{2}(\omega)||_{V}+\sum \limits_{m=1}\limits^{\infty}\sum \limits_{k =1}\limits^{m}||\eta_{1,k}(\omega)-\eta_{2,k}(\omega)||_{V}\id_{[\alpha_{m}(\omega),\alpha_{m+1}(\omega))}(t),
\end{align}
for all $t \geq 0$ and $\omega \in \Omega \setminus M$.\\
To this end, let $\omega \in \Omega \setminus M$ be arbitrary but fixed and introduce 
\begin{align*}
\hat{\alpha}_{0}:=0,~\hat{\alpha}_{m}:=\alpha_{m}(\omega),~\hat{\eta}_{i,m}:=\eta_{i,m}(\omega),~\hat{x}_{i}:=x_{i}(\omega)\text{ and }\hat{X}_{i}(t):=X_{i}(t,\omega),
\end{align*}
for all $t\geq 0$, $m \in \mathbb{N}$ and $i=1,2$.\\
Let us start tackling the task ahead of us, by proving that 
\begin{align}
\label{prop_uniqueproofeq2}
\lim \limits_{\varepsilon\searrow 0}\hat{X}_{i}(\hat{\alpha}_{\tilde{m}})-\hat{X}_{i}(\hat{\alpha}_{\tilde{m}}-\varepsilon)=\hat{\eta}_{i,\tilde{m}},~\forall \tilde{m} \in \mathbb{N}\text{ and }i=1,2.
\end{align}
in norm. Let $\tilde{m} \in \mathbb{N}$ and $i \in \{1,2\}$ be arbitrary but fixed and note that v) yields
\begin{align*}
\lim \limits_{\varepsilon\searrow 0} \int \limits_{0}\limits^{\hat{\alpha}_{\tilde{m}}-\varepsilon}\langle\Psi,\A^{\circ}\hat{X}_{i}(\tau)\rangle_{V}d\tau-\int \limits_{0}\limits^{\hat{\alpha}_{\tilde{m}}}\langle\Psi,\A^{\circ}\hat{X}_{i}(\tau)\rangle_{V}d\tau=0,~\forall \Psi \in V^{\ast}.
\end{align*}
Consequently, we get by invoking vi) that
\begin{align*}
\lim \limits_{\varepsilon\searrow 0}\langle\Psi, \hat{X}_{i}(\hat{\alpha}_{\tilde{m}})-\hat{X}_{i}(\hat{\alpha}_{\tilde{m}}-\varepsilon)\rangle_{V} = \sum \limits_{k =1 }\limits^{\tilde{m}}\langle\Psi,\hat{\eta}_{i,k}\rangle_{V}- \sum \limits_{k =1 }\limits^{\tilde{m}-1}\langle\Psi,\hat{\eta}_{i,k}\rangle_{V}=\langle\Psi,\hat{\eta}_{i,\tilde{m}}\rangle_{V},
\end{align*}
for all $\Psi \in V^{\ast}$. Moreover, ii) implies that there is a $u \in V$ such that
\begin{align}
\label{prop_uniqueproofeq3}
\lim \limits_{\varepsilon\searrow 0} ||\hat{X}_{i}(\hat{\alpha}_{\tilde{m}})-\hat{X}_{i}(\hat{\alpha}_{\tilde{m}}-\varepsilon) - u||_{V}=0.
\end{align}
Consequently, as convergence in norm implies weak convergence, we have
\begin{align*}
\langle\Psi,\hat{\eta}_{i,\tilde{m}}-u\rangle_{V}=\lim \limits_{\varepsilon\searrow 0}\langle\Psi, \hat{X}_{i}(\hat{\alpha}_{\tilde{m}})-\hat{X}_{i}(\hat{\alpha}_{\tilde{m}}-\varepsilon)\rangle_{V}-\langle\Psi, \hat{X}_{i}(\hat{\alpha}_{\tilde{m}})-\hat{X}_{i}(\hat{\alpha}_{\tilde{m}}-\varepsilon)\rangle_{V}=0,~\forall\Psi \in V^{\ast},
\end{align*}
which yields $\hat{\eta}_{i,\tilde{m}}=u$, since $V^{\ast}$ separates points. Consequently, (\ref{prop_uniqueproofeq3}) implies (\ref{prop_uniqueproofeq2}).\\

We will proceed by proving that
\begin{align}
\label{prop_uniqueproofeq4}
||\hat{X}_{1}(t)-\hat{X}_{2}(t)||_{V} \leq ||\hat{X}_{1}(\hat{\alpha}_{m})-\hat{X}_{2}(\hat{\alpha}_{m})||_{V},~\forall m \in \mathbb{N}_{0},~t \in [\hat{\alpha}_{m},\hat{\alpha}_{m+1}).
\end{align}
Proving (\ref{prop_uniqueproofeq4}) is divided into several intermediate steps and requires some notations. To this end, fix $m \in \mathbb{N}_{0}$, and introduce $\varepsilon \in (0,\hat{\alpha}_{m+1}-\hat{\alpha}_{m})$ arbitrary but fixed, $b_{\varepsilon}:=\hat{\alpha}_{m+1}-\hat{\alpha}_{m}-\varepsilon$, $F_{i}:[0,b_{\varepsilon}]\rightarrow V$ by $F_{i}:=\hat{X}_{i}(\cdot+\hat{\alpha}_{m})$ and $u_{i}:=\hat{X}_{i}(\hat{\alpha}_{m})$ for $i=1,2$.\\
Firstly, note that
\begin{align}
\label{prop_uniqueproofeq5}
F_{i}|_{[a,b]} \in W^{1,1}([a,b];V),~\forall [a,b] \subseteq (0,b_{\varepsilon}),~i \in \{1,2\},
\end{align}
since: For $[a,b] \subseteq (0,b_{\varepsilon})$ and $t \in [a,b]$ we have
\begin{align*}
\hat{\alpha}_{m}<\hat{\alpha}_{m}+a \leq \hat{\alpha}_{m} + t \leq \hat{\alpha}_{m} +b < \hat{\alpha}_{m+1}-\varepsilon,
\end{align*}
which yields by appealing to iv) that $\hat{X}_{i} \in W^{1,1}([\hat{\alpha}_{m}+a,\hat{\alpha}_{m}+b];V)$; and hence (\ref{prop_uniqueproofeq5}), by the definition of $F_{i}$.\\
Secondly, we will prove that
\begin{align}
\label{prop_uniqueproofeq6}
F_{i} \in C([0,b_{\varepsilon}];V),~i \in \{1,2\}.
\end{align}
Note that (\ref{prop_uniqueproofeq5}) already yields $F_{i} \in C((0,b_{\varepsilon});V)$ which then gives $F_{i} \in C([0,b_{\varepsilon});V)$, since $F_{i}$ inherits the right-continuity of $\hat{X}_{i}$. Consequently, (\ref{prop_uniqueproofeq6}) follows if $F_{i}$ is left-continuous in $b_{\varepsilon}$. As $\hat{X}_{i}$ is c\`{a}dl\`{a}g, we have that there is a $w_{i} \in V$ such that
\begin{align*}
\lim \limits_{h \searrow 0} F_{i}(b_{\varepsilon}-h)-F_{i}(b_{\varepsilon})=w_{i}.
\end{align*}
in norm. Moreover, note that $b_{\varepsilon}+\hat{\alpha}_{m}\in (\hat{\alpha}_{m},\hat{\alpha}_{m+1})$, which yields by invoking vi) that
\begin{align*}
\langle w_{i},\Psi\rangle_{V} = \lim \limits_{h \searrow 0} \int \limits_{0}\limits^{b_{\varepsilon}+\hat{\alpha}_{m}}\langle\Psi,\A^{\circ}X_{i}(\tau,\omega)\rangle_{V}d\tau-\int \limits_{0}\limits^{b_{\varepsilon}+\hat{\alpha}_{m}-h}\langle\Psi,\A^{\circ}X_{i}(\tau,\omega)\rangle_{V}d\tau = 0
\end{align*} 
for all $\Psi \in V^{\ast}$. As $V^{\ast}$ separates points, this is only possible if $w_{i}=0$, which establishes the desired left continuity and (\ref{prop_uniqueproofeq6}) follows.\\
The last intermediate step necessary to prove (\ref{prop_uniqueproofeq4}) is
\begin{align}
\label{prop_uniqueproofeq7}
0 \in F^{\prime}_{i}(t)+\A F_{i}(t),~\text{ a.e. } t \in (0,b_{\varepsilon}),~ F_{i}(0)=u_{i},~i=1,2.
\end{align}
To this end, note that (\ref{prop_uniqueproofeq5}) yields that there is for each $i \in \{1,2\}$ a Lebesgue null-set $M(F_{i})\subseteq(0,b_{\varepsilon})$ such that $F_{i}$ is differentiable (in norm) on $(0,b_{\varepsilon})\setminus M(F_{i})$.\\
Now let $V^{\ast}_{c}\subseteq V^{\ast}$ be a countable set which separates points; such a set exists due to \cite[Lemma 2.1 and Theorem 2.1]{pointsep}.\\ 
By virtue of the fundamental theorem of calculus for Lebesgue integrals, there is for each $\Psi \in V^{\ast}_{c}$ and $i \in \{1,2\}$ a Lebesgue null-set $M(\Psi,i)\subseteq (0,b_{\varepsilon})$ such that
\begin{align*}
\lim \limits_{h \rightarrow 0} \frac{1}{h} \left(\int \limits_{0}\limits^{t+\hat{\alpha}_{m}+h}\langle\Psi,\A^{\circ}\hat{X}_{i}(\tau)\rangle_{V}d\tau-\int \limits_{0}\limits^{t+\hat{\alpha}_{m}}\langle\Psi,\A^{\circ}\hat{X}_{i}(\tau)\rangle_{V}d\tau\right)=\langle\Psi,\A^{\circ}\hat{X}_{i}(t+\hat{\alpha}_{m})\rangle_{V}
\end{align*}
for all $t \in (0,b_{\varepsilon})\setminus M(\Psi,i)$, $i=1,2$.\\
Consequently, employing the previous equality, the differentiability a.e. of $F_{i}$ and vi) yields
\begin{eqnarray*}
	& & ~ 
\langle F_{i}^{\prime}(t)+\A^{\circ}F_{i}(t),\Psi\rangle_{V}\\
& = & ~ \lim \limits_{h \rightarrow 0}\frac{1}{h}\langle F_{i}(t+h)-F_{i}(t),\Psi \rangle_{V}+\langle \A^{\circ}F_{i}(t),\Psi\rangle_{V} \\
& = & ~ -\lim \limits_{h \rightarrow 0}\frac{1}{h} \left(\int \limits_{0}\limits^{t+\hat{\alpha}_{m}+h}\langle\Psi,\A^{\circ}\hat{X}_{i}(\tau)\rangle_{V}d\tau-\int \limits_{0}\limits^{t+\hat{\alpha}_{m}}\langle\Psi,\A^{\circ}\hat{X}_{i}(\tau)\rangle_{V}d\tau\right)+\langle \A^{\circ}F_{i}(t),\Psi\rangle_{V} \\
& = & ~ -\langle\Psi,\A^{\circ}\hat{X}_{i}(t+\hat{\alpha}_{m})\rangle_{V}+\langle \A^{\circ}F_{i}(t),\Psi\rangle_{V}\\
& = & 0
\end{eqnarray*}
for all $\Psi \in V^{\ast}_{c}$, $i \in \{1,2\}$ and $t \in (0,b_{\varepsilon})\setminus ( M(F_{i})\cup M(\Psi,i))$.\\
Now let $M_{i}:= \bigcup\limits_{\Psi \in V^{\ast}_{c}}M(\Psi,i) \cup  M(F_{i}) $, which is still a Lebesgue null-set since $V^{\ast}_{c}$ is countable. Then the previous calculation implies $\langle F_{i}^{\prime}(t)+\A^{\circ}F_{i}(t),\Psi\rangle_{V} = 0$ for all $\Psi \in V^{\ast}_{c}$ and $t \in (0,b_{\varepsilon})\setminus M_{i}$. Consequently, as $V^{\ast}_{c}$ separates points, we get $0=F_{i}^{\prime}(t)+\A^{\circ}F_{i}(t)$ for every $t \in (0,b_{\varepsilon})\setminus M_{i}$. Finally, iii) yields $F_{i}(t) \in D(\A)$ for all $t \in (0,b_{\varepsilon})$ and consequently $\A^{\circ}F_{i}(t)\in \A F_{i}(t)$ for all $t \in (0,b_{\varepsilon})$. Combining this with $0=F_{i}^{\prime}(t)+\A^{\circ}F(t)$ for every $t \in (0,b_{\varepsilon})\setminus M_{i}$ gives (\ref{prop_uniqueproofeq7}).\\
The results (\ref{prop_uniqueproofeq5})-(\ref{prop_uniqueproofeq7}) enable us to prove (\ref{prop_uniqueproofeq4}). By  (\ref{prop_uniqueproofeq5})-(\ref{prop_uniqueproofeq7}), we have that $F_{i}$ is a strong solution (cf. \cite[Definition 1.2]{BenilanBook}) of the initial value problem
\begin{align}
\label{prop_uniqueproofeq8}
0 \in U^{\prime}_{i}(t)+\A U(t),~\text{ a.e. } t \in (0,b_{\varepsilon}),~ U(0)=u_{i},
\end{align}
for $i=1,2$. Consequently $F_{i}$ is also a mild solution of (\ref{prop_uniqueproofeq8}), cf. \cite[Theorem 1.4]{BenilanBook}. Moreover, as $\A$ is m-accretive and densely defined (\ref{prop_uniqueproofeq8}) has precisely one mild solution, cf. \cite[Prop. 3.7]{BenilanBook}. This necessarily implies $F_{i}(t)=T_{\A}(t)u_{i}$ for $t \in [0,b_{\varepsilon}]$ and $i=1,2$. Therefore invoking Remark \ref{remark_ms} yields
\begin{align*}
||\hat{X}_{1}(t+\hat{\alpha}_{m})-\hat{X}_{2}(t+\hat{\alpha}_{m})||_{V}=||T_{\A}(t)u_{1}-T_{\A}(t)u_{2}||_{V}\leq||u_{1}-u_{2}||_{V}=||\hat{X}_{1}(\hat{\alpha}_{m})-\hat{X}_{2}(\hat{\alpha}_{m})||_{V},
\end{align*}
for all $t \in [0,b_{\varepsilon}]=[0,\hat{\alpha}_{m+1}-\hat{\alpha}_{m}-\varepsilon]$. As $\varepsilon \in (0,\hat{\alpha}_{m+1}-\hat{\alpha}_{m})$ can be chosen arbitrarily small, this holds for all $t \in [0,\hat{\alpha}_{m+1}-\hat{\alpha}_{m})$ which proves (\ref{prop_uniqueproofeq4}).\\

The next (and last) intermediate step enables us to prove the claim and reads as follows: For all $m \in \mathbb{N}$, all $t \in [\hat{\alpha}_{m},\hat{\alpha}_{m+1})$ and all $\varepsilon \in (0,\min(\hat{\alpha}_{1}-\hat{\alpha}_{0},..,\hat{\alpha}_{m}-\hat{\alpha}_{m-1}))$, we have
\begin{align}
\label{prop_uniqueproofeq9}
||\hat{X}_{1}(t)-\hat{X}_{2}(t)||_{V} \leq ||\hat{x}_{1}-\hat{x}_{2}||_{V}+ \sum \limits_{k =1 } \limits^{m}||\hat{X}_{1}(\hat{\alpha}_{k})-\hat{X}_{1}(\hat{\alpha}_{k}-\varepsilon)-\hat{X}_{2}(\hat{\alpha}_{k})+\hat{X}_{2}(\hat{\alpha}_{k}-\varepsilon)||_{V}.
\end{align}
This will be proven inductively. Let $m=1$, $t \in [\hat{\alpha}_{1},\hat{\alpha}_{2})$ and $\varepsilon \in (0,\hat{\alpha}_{1}-\hat{\alpha}_{0})$. Then appealing to (\ref{prop_uniqueproofeq4}) and i) yields
\begin{eqnarray*}
||\hat{X}_{1}(t)-\hat{X}_{2}(t)||_{V}
& \leq & ~ ||\hat{X}_{1}(\hat{\alpha}_{1})-\hat{X}_{2}(\hat{\alpha}_{1})||_{V}\\ 
& \leq & ~||\hat{x}_{1}-\hat{x}_{2}||_{V}+ ||\hat{X}_{1}(\hat{\alpha}_{1})-\hat{X}_{1}(\hat{\alpha}_{1}-\varepsilon)-\hat{X}_{2}(\hat{\alpha}_{1})+\hat{X}_{2}(\hat{\alpha}_{1}-\varepsilon)||_{V}.
\end{eqnarray*}
Induction step: Let $t \in [\hat{\alpha}_{m+1},\hat{\alpha}_{m+2})$ and $\varepsilon \in (0,\min(\hat{\alpha}_{1}-\hat{\alpha}_{0},..,\hat{\alpha}_{m+1}-\hat{\alpha}_{m}))$. Firstly, note that $\hat{\alpha}_{m+1}-\varepsilon \in [\hat{\alpha}_{m},\hat{\alpha}_{m+1})$ and that particularly $\varepsilon \in (0,\min(\hat{\alpha}_{1}-\hat{\alpha}_{0},..,\hat{\alpha}_{m}-\hat{\alpha}_{m-1}))$. Consequently, the induction hypothesis yields that
\begin{align*}
||\hat{X}_{1}(\hat{\alpha}_{m+1}-\varepsilon)-\hat{X}_{2}(\hat{\alpha}_{m+1}-\varepsilon)||_{V} \leq ||\hat{x}_{1}-\hat{x}_{2}||_{V}+ \sum \limits_{k =1 } \limits^{m}||\hat{X}_{1}(\hat{\alpha}_{k})-\hat{X}_{1}(\hat{\alpha}_{k}-\varepsilon)-\hat{X}_{2}(\hat{\alpha}_{k})+\hat{X}_{2}(\hat{\alpha}_{k}-\varepsilon)||_{V}
\end{align*}
Conclusively, appealing to (\ref{prop_uniqueproofeq4}), the triangle inequality and the preceding estimate gives
\begin{align*} 
||\hat{X}_{1}(t)-\hat{X}_{2}(t)||_{V} \leq ||\hat{x}_{1}-\hat{x}_{2}||_{V}+ \sum \limits_{k =1 } \limits^{m+1}||\hat{X}_{1}(\hat{\alpha}_{k})-\hat{X}_{1}(\hat{\alpha}_{k}-\varepsilon)-\hat{X}_{2}(\hat{\alpha}_{k})+\hat{X}_{2}(\hat{\alpha}_{k}-\varepsilon)||_{V},
\end{align*}
which implies (\ref{prop_uniqueproofeq9}).\\
Now the (from here on short) proof the claim will be derived: If $t \in [0,\hat{\alpha}_{1})$ we have
\begin{align*}
||\hat{X}_{1}(t)-\hat{X}_{2}(t)||_{V} \leq ||\hat{x}_{1}-\hat{x}_{2}||_{V}=||\hat{x}_{1}-\hat{x}_{2}||_{V}+ \sum \limits_{m=1}\limits^{\infty}\sum \limits_{k =1}\limits^{m}||\hat{\eta}_{1,k}-\hat{\eta}_{2,k}||_{V}\id_{[\hat{\alpha}_{m},\hat{\alpha}_{m+1})}(t),
\end{align*}
by (\ref{prop_uniqueproofeq4}) and i). If $t \in [\hat{\alpha}_{1},\infty)$, then (\ref{prop_uniqueproofeq0}) yields that there is an $\tilde{m} \in \mathbb{N}$ such that $t \in [\hat{\alpha}_{\tilde{m}},\hat{\alpha}_{\tilde{m}+1})$. Finally, employing (\ref{prop_uniqueproofeq9}) and (\ref{prop_uniqueproofeq2}) gives
\begin{eqnarray*}
	||\hat{X}_{1}(t)-\hat{X}_{2}(t)||_{V}
	& \leq & ~ ||\hat{x}_{1}-\hat{x}_{2}||_{V}+\lim \limits_{\varepsilon \searrow 0} \sum \limits_{k =1 } \limits^{\tilde{m}}||\hat{X}_{1}(\hat{\alpha}_{k})-\hat{X}_{1}(\hat{\alpha}_{k}-\varepsilon)-\hat{X}_{2}(\hat{\alpha}_{k})+\hat{X}_{2}(\hat{\alpha}_{k}-\varepsilon)||_{V}\\
	& = & ~ ||\hat{x}_{1}-\hat{x}_{2}||_{V}+ \sum \limits_{k =1 } \limits^{\tilde{m}}||\hat{\eta}_{1,k}-\hat{\eta}_{2,k}||_{V}\\
	& = & ~ ||\hat{x}_{1}-\hat{x}_{2}||_{V}+ \sum \limits_{m=1}\limits^{\infty}\sum \limits_{k =1}\limits^{m}||\hat{\eta}_{1,k}-\hat{\eta}_{2,k}||_{V}\id_{[\hat{\alpha}_{m},\hat{\alpha}_{m+1})}(t),
\end{eqnarray*}
which concludes the proof.
\end{proof}

\begin{theorem}\label{theorem_unique}  Let $V^{\ast} \subseteq V^{\prime}$ be a set that separates points, $\eta_{1},~\eta_{2}\in \mathcal{M}((0,\infty)\times Z \times \Omega;V)$ and introduce $x_{1},~x_{2}\in \mathcal{M}(\Omega;V)$. Moreover, assume that $X_{i}:[0,\infty)\times \Omega\rightarrow V$ is a mild solution of (\ref{acprm})$\{x_{i},\eta_{i},V^{\ast}\}$ for $i=1,2$. Then we have
	\begin{align}
	\label{theorem_uniqueeq}
	||X_{1}(t)-X_{2}(t)||_{V}\leq ||x_{1}-x_{2}||_{V}+\int \limits_{(0,t]\times Z} ||\eta_{1}(\tau,z)-\eta_{2}(\tau,z)||_{V} N_{\Theta}(d\tau\otimes z),~\forall t \geq 0,
	\end{align}
	almost surely.
\end{theorem}
\begin{proof} Let $x_{i,m} \in \mathcal{M}(\Omega;V)$, $\eta_{i,m}\in \mathcal{M}((0,\infty)\times Z \times \Omega;V)$ and $X_{i,m}:\Omega \times [0,\infty)\rightarrow V$ be such that
\begin{enumerate}
	\item $X_{i,m}$ is a strong solution of (\ref{acprm})$\{x_{i,m},\eta_{i,m},V^{\ast}\}$ for all $m \in \mathbb{N}$ and $i \in \{1,2\}$,
	\item $\lim \limits_{m \rightarrow\infty} \sup \limits_{\tau \in [0,t]} ||X_{i,m}(\tau)-X_{i}(\tau)||_{V}=0$ for all $t> 0$, $i \in \{1,2\}$ almost surely,
	\item $\lim \limits_{m \rightarrow\infty}\int \limits_{(0,t]\times Z}||\eta_{i,m}(\tau,z)-\eta_{i}(\tau,z)||_{V}N_{\Theta}(d\tau\otimes z)=0$ for all $t> 0$, $i \in \{1,2\}$ almost surely, and
	\item $||X_{1,m}(t)-X_{2,m}(t)||_{V}\leq ||x_{1,m}-x_{2,m}||_{V}+\int \limits_{(0,t]\times Z} ||\eta_{1,m}(\tau,z)-\eta_{2,m}(\tau,z)||_{V} N_{\Theta}(d\tau\otimes z)$ for all $t\geq 0$, $m \in \mathbb{N}$, $i \in \{1,2\}$ almost surely.
\end{enumerate}
Proposition \ref{prop_unique} (and the definition of mild solution) guarantee the existence of these quantities. Consequently, we have
\begin{eqnarray*}
	& & ~
||X_{1}(t)-X_{2}(t)||_{V} \\
& = & ~ \lim \limits_{m \rightarrow \infty } ||X_{1,m}(t)-X_{2,m}(t)||_{V} \\
& \leq & ~ \lim \limits_{m \rightarrow \infty } ||x_{1,m}-x_{2,m}||_{V}+\int \limits_{(0,t]\times Z} ||\eta_{1,m}(\tau,z)-\eta_{2,m}(\tau,z)||_{V} N_{\Theta}(d\tau\otimes z)  \\
& = & ~  ||x_{1}-x_{2}||_{V}+\int \limits_{(0,t]\times Z} ||\eta_{1}(\tau,z)-\eta_{2}(\tau,z)||_{V} N_{\Theta}(d\tau\otimes z)  
\end{eqnarray*}
for all $t\geq 0$, with probability one.
\end{proof}

Theorem \ref{theorem_unique} has two important consequences: Uniqueness of mild solutions of (\ref{acprm}) and an upper bound for the solution.

\begin{corollary}\label{corollary_unique} (\ref{acprm}) has at most one mild solution; more precisely: Let $V^{\ast}\subseteq V^{\prime}$ be such that it separates points, let $x \in \mathcal{M}(\Omega;V)$, $\eta \in \mathcal{M}((0,\infty)\times Z \times \Omega ; V)$ and assume that $X_{1},~X_{2}:[0,\infty)\times \Omega \rightarrow V$ are mild solutions of (\ref{acprm})$\{x,\eta,V^{\ast}\}$, then $X_{1}$ and $X_{2}$ are indistinguishable.
\end{corollary}

\begin{theorem}\label{theorem_bound} Let $V^{\ast}\subseteq V^{\prime}$ be such that it separates points, let $x \in \mathcal{M}(\Omega;V)$, $\eta \in \mathcal{M}((0,\infty)\times Z \times \Omega ; V)$ and assume that $X:[0,\infty)\times \Omega \rightarrow V$ is a mild solution of (\ref{acprm})$\{x,\eta,V^{\ast}\}$. Finally, assume that $(0,0) \in \mathcal{A}0$. Then we have
\begin{align}
\label{theorem_boundeq}
||X(t)||_{V}\leq ||x||_{V}+\int \limits_{(0,t]\times Z} ||\eta(\tau,z)||_{V} N_{\Theta}(d\tau\otimes z),~\forall t \geq 0,
\end{align}
almost surely.
\end{theorem}
\begin{proof} If $0 \in \mathcal{A}0$, then it is plain that $T_{\A}(t)0=0$ for all $t\geq 0$. Consequently, we have $\mathcal{A}^{\circ}0=0$. This implies that the stochastic process which is constantly zero, is a strong (and therefore also mild) solution of (\ref{acprm})$\{0,0,V^{\ast}\}$. Consequently, the claim follows from Theorem \ref{theorem_unique}.
\end{proof}

Due to the nonlinearity it is generally not true that $0 \in D(\A)$. Moreover, we shall see that (\ref{theorem_boundeq}) cannot be improved without additional assumptions; by that we mean that for the $p$-Laplacian example considered in the next section, we will find a nontrivial drift $\eta$ and a nontrivial initial $x$, such that the inequality in (\ref{theorem_boundeq}) turns into an equality, see Theorem \ref{theorem_plaplacebound}.\\

Now we will turn to the question of existence. To this end, some preparatory lemmas are in order:

\begin{lemma}\label{lemma_ll} Let $v \in V$ be arbitrary but fixed. Then $T_{\A}(\cdot)v$ is locally Lipschitz continuous on $(0,\infty)$ and differentiable almost everywhere. Moreover, we have
\begin{align}
\label{lemma_sseq}
-\A^{\circ}T_{\A}(t)v=T^{\prime}_{\A}(t)v,~\text{a.e. } t \in (0,\infty).
\end{align}
\end{lemma}
\begin{proof}  Let $m \in \mathbb{N}\setminus\{1\}$. We will prove that $T_{\A}(\cdot)v|_{[\frac{1}{m},m]}$ is Lipschitz continuous, differentiable almost everywhere and that (\ref{lemma_sseq}) holds a.e. on $(\frac{1}{m},m)$; which obviously implies the claim.\\
For the sake of space, introduce $\zeta :[\frac{1}{m},m]\rightarrow V$ by $\zeta(t):=T_{\A}(t)v$. Moreover, set $u:=T_{\A}(\frac{1}{m})v$ and note that $u \in D(\A)$, since $T_{\A}$ is assumed to be domain invariant. Consequently, \cite[Lemma 7.8]{BenilanBook} yields that there is a constant $L>0$ such that
\begin{align*}
||T_{\A}(\tau_{1})u-T_{\A}(\tau_{2})u||_{V} \leq L ||\tau_{1}-\tau_{2}||_{V},~\forall \tau_{1},~\tau_{2}\in [0,m-\frac{1}{m}].
\end{align*}
Hence we get
\begin{align}
\label{lemma_ssproofeq0}
||\zeta(t_{1})-\zeta(t_{2})||_{V} = ||T_{\A}(t_{1}-\frac{1}{m})u-T_{\A}(t_{2}-\frac{1}{m})u||_{V} \leq L|t_{1}-t_{2}|,~\forall t_{1},~t_{2}\in [\frac{1}{m},m],
\end{align}
which gives the desired Lipschitz continuity.\\
Moreover, invoking (\ref{def_geneq}) and having in mind the domain invariance yields
\begin{align}
\label{lemma_ssproofeq1}
\lim \limits_{h \searrow 0} \frac{\zeta(t+h)-\zeta(t)}{h}=\frac{T_{\A}(h)T_{\A}(t)v-T_{\A}(t)v}{h}=-\A^{\circ}T_{\A}(t)v=-\A^{\circ}\zeta(t)=:\zeta^{\prime}_{r}(t),~\forall t \in [\frac{1}{m},m).
\end{align}
Consequently, $\zeta$ is (everywhere) right differentiable. Moreover, the preceding equation yields that it remains to prove that it is also differentiable almost everywhere.\\
To this end, note that Remark \ref{remark_ms}.ii) implies that $\zeta$ is continuous. Consequently, it is a fortiori  (w.r.t. the Lebesgue measure) strongly measurable, see \cite[Corollary 1.1.2.c)]{greenbook}. Moreover, the continuity also yields that
	\begin{align*}
	\int \limits_{\frac{1}{m}}\limits^{m} ||\zeta(t)||_{V} dt < \infty
	\end{align*}
	and therefore $\zeta \in L^{1}([\frac{1}{m},m];V)$, by \cite[Theorem 1.1.4]{greenbook}.\\
Moreover, (\ref{lemma_ssproofeq1}) implies that $\zeta^{\prime}_{r}$ is the pointwise limit of strongly measurable functions and therefore strongly measurable as well, see \cite[Corollary 1.1.2.d)]{greenbook}. In addition, (\ref{lemma_ssproofeq0}) yields that
\begin{align*}
	\int \limits_{\frac{1}{m}}\limits^{m}||\zeta^{\prime}_{r}(t)||_{V} dt \leq \int \limits_{\frac{1}{m}}\limits^{m}L d t =  L(m-\frac{1}{m})< \infty,
\end{align*}
which gives $\zeta^{\prime}_{r} \in L^{1}([\frac{1}{m},m];V)$.\\
	Now introduce $\zeta_{\ast}:[\frac{1}{m},m]\rightarrow V$, by
	\begin{align*}
	\zeta_{\ast}(t):= \int \limits_{\frac{1}{m}}\limits^{t} \zeta^{\prime}_{r}(\tau)d\tau+u,~\forall t \in [\frac{1}{m},m].
	\end{align*}
	Then the fundamental theorem of calculus for Bochner integrals (see \cite[Proposition 1.2.2]{greenbook}) yields that $\zeta_{\ast}$ is differentiable almost everywhere and that $\zeta_{\ast}^{\prime}(t)=\zeta^{\prime}_{r}(t)$ for a.e. $t \in [\frac{1}{m},m]$.\\
	Consequently, the claim follows if $\zeta(t)=\zeta_{\ast}(t)$ for every $t \in [\frac{1}{m},m]$.\\
	To prove this, introduce $\Gamma:[\frac{1}{m},m] \rightarrow \mathbb{R}$ by
	\begin{align*}
	\Gamma(t):= ||\zeta(t)-\zeta_{\ast}(t)||_{V},~\forall t \in [\frac{1}{m},m].
	\end{align*}
	Firstly, note that obviously $\Gamma(\frac{1}{m})=0$. Moreover, we have
	\begin{eqnarray*}
		& & ~
		\lim \limits_{h \searrow  0}   \left|\frac{\Gamma(t+h)-\Gamma(t)}{h}\right| \\
		& \leq  & ~ \lim \limits_{h \searrow  0} \left(\left| \left| \frac{\zeta(t+h)-\zeta(t)}{h}-\zeta^{\prime}_{r}(t) \right|\right|_{V} + \left| \left| \frac{-\zeta_{\ast}(t+h)+\zeta_{\ast}(t)}{h}+\zeta^{\prime}_{r}(t) \right|\right|_{V}\right) \\ 
		& = & ~ 0,
	\end{eqnarray*}
	for almost every $t \in [\frac{1}{m},m)$, i.e.  $\Gamma$ is almost everywhere right differentiable and the right derivative is equal to zero.\\
	In addition, one has by invoking Lemma (\ref{lemma_ssproofeq0}) that
	\begin{align*}
	|\Gamma(t+h)-\Gamma(t)| \leq Lt+\left|\left|\int\limits_{t}\limits^{t+h}\zeta^{\prime}_{r}(\tau)d\tau\right|\right|\leq Lt+\int\limits_{t}\limits^{t+h}||\zeta^{\prime}_{r}(\tau)||d\tau \leq 2Lt
	\end{align*}
	for all $t \in [\frac{1}{m},m]$ and $0 < h \leq m-t$.\\
	Conclusively, the last estimate yields that $\Gamma$ is Lipschitz continuous, which implies, as $\mathbb{R}$ has the Radon-Nikodym property, that it is differentiable almost everywhere. Since the right derivate of $\Gamma$ is zero almost everywhere, the almost everywhere derivative is also zero a.e. Finally, the Lipschitz continuity of $\Gamma$ yields that $\Gamma$ is constant, and hence $\Gamma(t)=0$ for all $t \in [\frac{1}{m},m]$.
\end{proof}

\begin{lemma}\label{lemma_intcalcu} Let $t> 0$, $v \in V$ and $\Psi \in V^{\prime}$. Moreover, assume that $\langle\Psi,\A^{\circ}T_{\A}(\cdot)v\rangle_{V}\in L^{1}((0,t))$. Then we have
\begin{align*}
\int \limits_{0}\limits^{t}\langle\Psi,\A^{\circ}T_{\A}(\tau)v\rangle_{V}d\tau=-\langle \Psi, T_{\A}(t)v-v\rangle_{V}.
\end{align*}
\end{lemma}
\begin{proof} Let $\varepsilon\in (0,t)$ be arbitrary but fixed. Firstly, Lemma \ref{lemma_ll} obviously implies that the mapping $(\varepsilon,t)\ni \tau \mapsto \langle\Psi,T_{\A}(\tau)v\rangle_{V}$ is Lipschitz continuous and differentiable almost everywhere with
\begin{align*}
\frac{\partial}{\partial \tau}\langle\Psi,T_{\A}(\tau)v\rangle_{V} = - \langle\Psi,\A^{\circ}T_{\A}(\tau)v\rangle_{V},~\text{a.e. } \tau \in (\varepsilon,t). 
\end{align*}
Consequently, we have
\begin{align}
\label{lemma_prooofinteq1}
\int \limits_{\varepsilon}\limits^{t}\langle\Psi,\A^{\circ}T_{\A}(\tau)v\rangle_{V}d\tau=-\langle \Psi, T_{\A}(t)v-T_{\A}(\varepsilon)v\rangle_{V}
\end{align}
Now the claim follows from (\ref{lemma_prooofinteq1}) by taking limit, more precisely: We have
\begin{align*}
\lim \limits_{\varepsilon\searrow 0} -\langle \Psi, T_{\A}(t)v-T_{\A}(\varepsilon)v\rangle_{V} = -\langle \Psi, T_{\A}(t)v-v\rangle_{V},
\end{align*}
by Remark \ref{remark_ms}.ii). Moreover, dominated convergence yields that
\begin{align*}
\int \limits_{0}\limits^{t}\langle\Psi,\A^{\circ}T_{\A}(\tau)v\rangle_{V}d\tau = \int \limits_{0}\limits^{t}\lim \limits_{\varepsilon\searrow 0}\langle\Psi,\A^{\circ}T_{\A}(\tau)v\rangle_{V}\id_{(\varepsilon,t)}(\tau)d\tau = \lim \limits_{\varepsilon\searrow 0} \int \limits_{\varepsilon}\limits^{t}\langle\Psi,\A^{\circ}T_{\A}(\tau)v\rangle_{V}d\tau,
\end{align*}
which is applicable since $\langle\Psi,\A^{\circ}T_{\A}(\cdot)v\rangle_{V}\in L^{1}((0,t))$ by assumption.
\end{proof}

The stochastic process introduced in the following definition will turn out to solve (\ref{acprm}).

\begin{definition} Let $\eta \in \mathcal{M}((0,\infty)\times Z \times \Omega;V)$ and $x\in \mathcal{M}(\Omega;V)$. Moreover, introduce \linebreak$\eta_{m}(\omega):=\eta(\alpha_{m}(\omega),\Theta(\omega)(\alpha_{m}(\omega),\omega)$ for all $m \in \mathbb{N}$ and $\P$-a.e. $\omega \in \Omega$. In addition, let $\x_{m}:\Omega \rightarrow V$ for all $m \in \mathbb{N}_{0}$, be defined by $\x_{0}:=x$ and 
\begin{align*}
\x_{m}:=T_{\A}(\alpha_{m}-\alpha_{m-1})\x_{m-1}+\eta_{m},~\forall m \in \mathbb{N}.
\end{align*}
Finally, introduce $\X:[0,\infty)\times \Omega\rightarrow V$ by 
\begin{align}
\label{def_soleq}
\X(t):= \sum \limits_{m =0}\limits^{\infty} T_{\A}((t-\alpha_{m})_{+})(\x_{m})\id_{[\alpha_{m},\alpha_{m+1})}(t),~\forall t \geq 0.
\end{align}	
The sequence $(\x_{m})_{m\in \mathbb{N}_{0}}$ will be called the sequence of jumps generated by $(x,\eta)$ and $\X$ will be called the process generated by $(x,\eta)$.
\end{definition}

\begin{remark} Note that $0=\alpha_{0}<\alpha_{1}<\alpha_{2}<...$ up to a $\P$-null-set. Consequently, the right-hand-side series in (\ref{def_soleq}) simply reduces (for almost all $\omega \in \Omega$) to a single summand, which ensures that $\X$ is well-defined.
\end{remark}

\begin{lemma}\label{lemma_solproperties} Let $x\in \mathcal{M}(\Omega;V)$, $\eta \in \mathcal{M}((0,\infty)\times Z \times \Omega;V)$ and let $\X$ be the process generated by $(x,\eta)$. Then $\X$ is a $\B([0,\infty))\otimes\F$-$\B(V)$-measurable stochastic process\footnote{Actually, we only prove that there is a process indistinguishable of $\X$ which is $\B([0,\infty))\otimes\F$-$\B(V)$-measurable. We follow the common mathematical convention of identifying indistinguishable processes with each other.}, which fulfills the following assertions for $\P$-a.e. $\omega \in \Omega$. 
\begin{enumerate}
	\item $\X(0,\omega)=x(\omega)$,
	\item the mapping $[0,\infty)\ni t \mapsto \X(t,\omega)$ is c\`{a}dl\`{a}g,
	\item $\X(t,\omega)\in D(\mathcal{A}),~\forall t \in (0,\infty)\setminus\{\alpha_{m}(\omega):m\in \mathbb{N}\}$, 
	\item $\forall m \in \mathbb{N}_{0},~\forall [a,b]\subseteq (\alpha_{m}(\omega),\alpha_{m+1}(\omega)):~ \X(\cdot,\omega)|_{[a,b]} \in W^{1,1}([a,b];V)$ and
	\item $\A^{\circ}\X(\cdot,\omega)$ is $\B((0,\infty))-\B(V)$-measurable.
\end{enumerate}
\end{lemma}
\begin{proof} Firstly, introduce $\eta_{m}(\omega):=\eta(\alpha_{m}(\omega),\Theta(\omega)(\alpha_{m}(\omega),\omega)$ for all $m \in \mathbb{N}$ and $\P$-a.e. $\omega \in \Omega$. In addition, let $(\x_{m})_{m\in \mathbb{N}_{0}}$ be the sequence of jumps generated by $(x,\eta)$.\\ 
We will start by showing that $\X(t)\in \mathcal{M}(\Omega;V)$, i.e. that $\X$ is indeed a stochastic process. To this end, note that we already know that each $\eta_{m}$ is $\F-\B(V)$-measurable, that each $\alpha_{m}$ is $\F-\B(\R)$-measurable and that $\x_{0}$ is $\F-\B(V)$-measurable, which yields (by a simple induction) that each $\x_{m}$ is $\F-\B(V)$-measurable, since $T_{\A}:[0,\infty)\times V \rightarrow V$ is jointly continuous (see Remark \ref{remark_ms}) and $\mathfrak{B}([0,\infty)\times V)= \mathfrak{B}([0,\infty))\otimes \mathfrak{B}( V)$, see \cite[page 244]{Billingsley}.\\
It is now plain to verify that $\X(t)$ is $\F-\B(V)$-measurable, for all $t\geq 0$. The desired joint measurability will be established, once we have proven that $\X$ has almost surely c\`{a}dl\`{a}g paths.\\
Now let $M\in \F$ is a $\P$-null-set such that $0=\alpha_{0}(\omega)<\alpha_{1}(\omega)<\alpha_{2}(\omega)<...$ as well as $\lim \limits_{m \rightarrow\infty}\alpha_{m}(\omega)=\infty$ and $\alpha_{m}(\omega)\in D(\Theta(\omega))$ for all $\omega \in \Omega \setminus M$ and $m \in \mathbb{N}$. And let us prove i)-v) for all $\omega \in \Omega \setminus M$.\\
So let $\omega \in \Omega \setminus M$ be arbitrary but fixed.\\
It is plain that $\X(0,\omega)=T_{\A}(0)\x_{0}(\omega)=\x_{0}(\omega)=x(\omega)$ which gives i).\\
Proof of ii). Let $t\geq 0$ be given. Then there is precisely one $m \in \mathbb{N}$ such that $t \in [\alpha_{m}(\omega),\alpha_{m+1}(\omega))$. Moreover, for $h\geq 0$ sufficiently small, we also have $t+h\in [\alpha_{m}(\omega),\alpha_{m+1}(\omega)) $ and hence
\begin{align*}
\lim \limits_{h \searrow 0} \X(t+h,\omega)-\X(t,\omega)=\lim \limits_{h \searrow 0}T_{\A}(t+h-\alpha_{m}(\omega))\x_{m}(\omega)-T_{\A}(t-\alpha_{m}(\omega))\x_{m}(\omega)=0,
\end{align*}
by Remark \ref{remark_ms}.ii), which gives the desired right continuity.\\ 
Moreover, if $t \in (\alpha_{m}(\omega),\alpha_{m+1}(\omega))$, it follows absolutely analogously that
\begin{align*}
\lim \limits_{h \searrow 0} \X(t-h,\omega)-\X(t,\omega)=0
\end{align*}
Finally, if $t=\alpha_{m}(\omega)$ for $m \in \mathbb{N}$, we have
\begin{align*}
\lim \limits_{h \searrow 0} \X(t-h,\omega)-\X(t,\omega)=T_{\A}(\alpha_{m}(\omega)-\alpha_{m-1}(\omega))\x_{m-1}(\omega)-\x_{m}(\omega)=-\eta_{m}(\omega),
\end{align*}
which gives the existence of the left limits.\\
Note the following: If $X:[0,\infty)\times \Omega \rightarrow V$ is defined by $X:=\X$ on $[0,\infty)\times (\Omega \setminus M)$ and $X:=0$ on $[0,\infty)\times M$, then $\X$ and $X$ are indistinguishable and $X$ is a stochastic process which is c\`{a}dl\`{a}g. Consequently, $X$ is $\B([0,\infty))\otimes \F$-$\B(V)$-measurable, see \cite[Proposition 2.2.3.2]{SIBS}.\\ 
Proof of iii). Let $t \in (0,\infty) \setminus \{\alpha_{m}(\omega):~m\in \mathbb{N}\}$, then there is precisely one $m \in \mathbb{N}_{0}$ such that $t \in (\alpha_{m}(\omega),\alpha_{m+1}(\omega))$ and therefore $\X(t,\omega)=T_{\A}(t-\alpha_{m}(\omega))\x_{m}(\omega)$. Consequently, as $T_{\A}$ is domain invariant, we have $\X(t,\omega)\in D(\A)$.\\
Proof of iv). Let $m \in \mathbb{N}$ and $[a,b]\subseteq (\alpha_{m}(\omega),\alpha_{m+1}(\omega))$. Then it is plain that\linebreak $\X(\cdot,\omega)|_{[a,b]}=T_{\A}(\cdot-\alpha_{m}(\omega))\x_{m}(\omega)$. But the Lipschitz continuity and differentiability almost everywhere of this mapping follows trivially from Lemma \ref{lemma_ll}.\\
Proof of v). Let $(h_{k})_{k \in \mathbb{N}}\subseteq (0,\infty)$ be a null sequence. Moreover, introduce $f_{k,m}:(0,\infty)\rightarrow V$ by
\begin{align*}
f_{k,m}(t):= \frac{T_{\A}((t-\alpha_{m}(\omega))_{+}+h_{k})\x_{m}(\omega)-T_{\A}((t-\alpha_{m}(\omega))_{+})\x_{m}(\omega)}{h_{k}}\id_{[\alpha_{m}(\omega),\alpha_{m+1}(\omega))}(t),
\end{align*}
for all $t \in (0,\infty)$,$~m \in \mathbb{N}_{0}$ and $k \in \mathbb{N}$.\\
Then we have
\begin{align}
\label{lemma_solpropertiesproof1}
\lim \limits_{k \rightarrow \infty} f_{k,m}(t)=-\A^{\circ}(T_{\A}((t-\alpha_{m}(\omega))_{+})\x_{m}(\omega))\id_{[\alpha_{m}(\omega),\alpha_{m+1}(\omega))}(t),
\end{align}
for all $m \in \mathbb{N}_{0}$ and $t \in (0,\infty)\setminus \{\alpha_{j}(\omega):j\in \mathbb{N}\}$, since: If $t \not \in [\alpha_{m}(\omega),\alpha_{m+1}(\omega))$, for a given $m \in \mathbb{N}_{0}$, then (\ref{lemma_solpropertiesproof1}) is trivial and if $t \in (\alpha_{m}(\omega),\alpha_{m+1}(\omega))$, we have by domain invariance of $T_{\A}$ that
\begin{eqnarray*}
\lim \limits_{k \rightarrow \infty} f_{k,m}(t)
& = & ~ \lim \limits_{k \rightarrow \infty} \frac{T_{\A}(h_{k})T_{\A}(t-\alpha_{m}(\omega))\x_{m}(\omega)-T_{\A}(t-\alpha_{m}(\omega))\x_{m}(\omega)}{h_{k}}\\
& = & ~ -\A^{\circ}T_{\A}(t-\alpha_{m}(\omega))\x_{m}(\omega).
\end{eqnarray*}
In addition, Remark \ref{remark_ms}.ii) gives that each $f_{k,m}$ is $\B((0,\infty))-\B(V)$-measurable. Consequently, (\ref{lemma_solpropertiesproof1}) yields that $(0,\infty)\ni t \mapsto -\A^{\circ}(T_{\A}((t-\alpha_{m}(\omega))_{+})\x_{m}(\omega))\id_{[\alpha_{m}(\omega),\alpha_{m+1}(\omega))}(t)$ is also $\B(0,\infty)-\B(V)$-measurable for all $m \in \mathbb{N}_{0}$, since it is (except for a countable set) the pointwise limit of $\B(0,\infty)-\B(V)$-measurable functions.\\
Finally, it is plain that
\begin{align*}
\A^{\circ}\X(t,\omega) = \sum \limits_{m=0}\limits^{\infty}\A^{\circ}(T_{\A}((t-\alpha_{m}(\omega))_{+})\x_{m})\id_{[\alpha_{m}(\omega),\alpha_{m+1}(\omega))}(t),~\forall t \in (0,\infty),
\end{align*}
which implies the desired measurability.
\end{proof} 

The preceding lemma enables us to give a condition ensuring that (\ref{acprm}) has a (uniquely determined) strong solution. Afterwards, just one more approximation lemma is needed to formulate this paper's central result: A criteria ensuring the existence of a unique mild solution of (\ref{acprm}).

\begin{proposition}\label{prop_existence} Let $\mathcal{V}\subseteq V$ be a subspace of $V$ and let $V^{\ast}\subseteq V^{\prime}$ be a subset which separates points. Moreover, let $x \in \mathcal{M}(\Omega;V)$, $\eta \in \mathcal{M}((0,\infty)\times Z \times \Omega;V)$ and let  $\X$ denote the process generated by $(x,\eta)$. In addition, assume that $x \in \mathcal{V}$ a.s. and $\eta(t,z)\in \mathcal{V}$ for all $t \in (0,\infty)$ and $z \in Z$ with probability one. Finally, assume that $\mathcal{V}$ is an invariant set w.r.t. $T_{\A}$ and that $\langle \Psi,\A^{\circ}T_{\A}(\cdot)u\rangle_{V}\in L^{1}((0,t))$ for all $t>0$, $u \in \mathcal{V}$ and $\Psi \in V^{\ast}$. Then the stochastic process $\X$ is the unique strong solution of (\ref{acprm})$\{x,\eta,V^{\ast}\}$.
\end{proposition}
\begin{proof} Firstly, introduce $\eta_{m}(\omega):=\eta(\alpha_{m}(\omega),\Theta(\omega)(\alpha_{m}(\omega),\omega)$ for all $m \in \mathbb{N}$ and $\P$-a.e. $\omega \in \Omega$. In addition, let $(\x_{m})_{m\in \mathbb{N}_{0}}$ be the sequence of jumps generated by $(x,\eta)$.\\ 
Combining Lemma \ref{lemma_solproperties} and Lemma \ref{lemmma_solequivalence} yields that $\X$ is a strong solution of (\ref{acprm})$\{x,\eta,V^{\ast}\}$, if 
\begin{align}
\label{prop_exisproofeq1}
\int \limits_{0}\limits^{t}|\langle \Psi,\A^{\circ}\X(\tau)\rangle_{V}|d\tau < \infty,~\forall t>0,~\Psi \in V^{\ast}
\end{align}	
a.s. and
\begin{align} 
\label{prop_exisproofeq2}
\langle\Psi,\X(t)-x\rangle_{V}+\int \limits_{0}\limits^{t}\langle\Psi,\A^{\circ}\X(\tau)\rangle_{V}d\tau=\sum \limits_{m =1}\limits^{\infty} \sum \limits_{k =1 }\limits^{m}\langle\Psi,\eta_{k}\rangle_{V}\id_{[\alpha_{m},\alpha_{m+1})}(t),~\forall t>0,~\Psi \in V^{\ast}
\end{align}
almost surely.\\
Moreover, by Corollary \ref{corollary_unique} we get that this strong solution is unique. (The Corollary is indeed applicable, since every strong solution is obviously also a mild one.)\\
Let $M \in \F$ be a $\P$-null-set such that $0=\alpha_{0}(\omega)<\alpha_{1}(\omega)<\alpha_{2}(\omega)<...$ as well as $\lim \limits_{m \rightarrow\infty}\alpha_{m}(\omega)=\infty$, $\alpha_{m}(\omega)\in D(\Theta(\omega))$ for all $m \in \mathbb{N}$, $\eta(t,z,\omega)\in \mathcal{V}$ for all $t \in (0,\infty)$ and $z \in Z$, $x(\omega)\in \mathcal{V}$ and such that Lemma \ref{lemma_solproperties}.i-v) hold for all $\omega \in \Omega \setminus M$.\\
Now the claims will be proven for all $\omega \in \Omega \setminus M$. To this end, fix one of these $\omega$ and let us start by proving inductively that 
\begin{align}
\label{prop_exisproofeq3}
\x_{m}(\omega)\in \mathcal{V},~\forall m \in \mathbb{N}_{0}.
\end{align} 
For $m=0$, we have $\x_{0}(\omega)=x(\omega)\in\mathcal{V}$. Moreover, for any $m \in \mathbb{N}$ we have $\eta_{m}(\omega) \in \mathcal{V}$. In addition, if $\x_{m-1}(\omega)\in\mathcal{V}$, then $T(\alpha_{m}(\omega)-\alpha_{m-1}(\omega))\x_{m-1}(\omega)\in \mathcal{V}$, since $\mathcal{V}$ is invariant w.r.t. $T_{\A}$ and \linebreak$\alpha_{m}(\omega)-\alpha_{m-1}(\omega)>0$. Consequently, we get $\x_{m}(\omega)\in \mathcal{V}$, since $\mathcal{V}$ is a vector space.\\
Proof of (\ref{prop_exisproofeq1}). For a given $t \in [0,\infty)$ there is an $m \in \mathbb{N}_{0}$ such that $t \in [\alpha_{m}(\omega),\alpha_{m+1}(\omega))$. This yields
\begin{eqnarray*}
\int \limits_{0}\limits^{t}|\langle \Psi,\A^{\circ}\X(\tau,\omega)\rangle_{V}|d\tau 
&\leq&~ \int \limits_{0}\limits^{\alpha_{m+1}(\omega)}|\langle \Psi,\A^{\circ}\X(\tau,\omega)\rangle_{V}|d\tau \\
& = & ~ \sum \limits_{k =0}\limits^{m} \int \limits_{\alpha_{k}(\omega)}\limits^{\alpha_{k+1}(\omega)}|\langle \Psi,\A^{\circ}T_{\A}(\tau-\alpha_{k}(\omega))\x_{k}(\omega)\rangle_{V}|d\tau \\
& = & ~ \sum \limits_{k =0}\limits^{m} \int \limits_{0}\limits^{\alpha_{k+1}(\omega)-\alpha_{k}(\omega)}|\langle \Psi,\A^{\circ}T_{\A}(\tau)\x_{k}(\omega)\rangle_{V}|d\tau.
\end{eqnarray*}
Moreover, invoking (\ref{prop_exisproofeq3}) gives 
\begin{align*}
\int \limits_{0}\limits^{\alpha_{k+1}(\omega)-\alpha_{k}(\omega)}|\langle \Psi,\A^{\circ}T_{\A}(\tau)\x_{k}(\omega)\rangle_{V}|d\tau < \infty
\end{align*}
for all $k=0,...,m$, which concludes the proof of (\ref{prop_exisproofeq1}).\\
Proof of (\ref{prop_exisproofeq2}). Let $t\in (0,\infty)$ and (as usually) let $m \in \mathbb{N}_{0}$ such that $t \in [\alpha_{m}(\omega),\alpha_{m+1}(\omega))$.\\
If $m=0$, we have $\sum \limits_{j =1}\limits^{\infty} \sum \limits_{k =1 }\limits^{j}\langle\Psi,\eta_{k}(\omega)\rangle_{V}\id_{[\alpha_{j}(\omega),\alpha_{j+1}(\omega))}(t)$=0 and $\X(t,\omega)=T_{\A}(t)x(\omega)$. Hence in this case (\ref{prop_exisproofeq2}) follows from Lemma \ref{lemma_intcalcu}, which is applicable since $x(\omega)\in \mathcal{V}$.\\ 
Now assume $m \in \mathbb{N}$. Then we have
\begin{eqnarray*}
	& & ~
\int \limits_{0}\limits^{t}\langle \Psi,\A^{\circ}\X(\tau,\omega)\rangle_{V}d\tau \\
& = &  ~ \sum \limits_{k =0}\limits^{m-1} \int \limits_{\alpha_{k}(\omega)}\limits^{\alpha_{k+1}(\omega)}\langle \Psi,\A^{\circ}\X(\tau,\omega)\rangle_{V}d\tau + \int \limits_{\alpha_{m}(\omega)}\limits^{t}\langle \Psi,\A^{\circ}\X(\tau,\omega)\rangle_{V}d\tau \\
& = &  ~ \sum \limits_{k =0}\limits^{m-1} \int \limits_{0}\limits^{\alpha_{k+1}(\omega)-\alpha_{k}(\omega)}\langle \Psi,\A^{\circ}T_{\A}(\tau)\x_{k}(\omega)\rangle_{V}d\tau + \int \limits_{0}\limits^{t-\alpha_{m}(\omega)}\langle \Psi,\A^{\circ}T_{\A}(\tau)\x_{m}(\omega)\rangle_{V}d\tau.
\end{eqnarray*}
In addition, (\ref{prop_exisproofeq3}) enables us to use Lemma \ref{lemma_intcalcu} now. Doing so, and having in mind that \linebreak$T_{\A}(\alpha_{k+1}(\omega)-\alpha_{k}(\omega))\x_{k}(\omega)=\x_{k+1}(\omega)-\eta_{k+1}(\omega)$ for all $k \in \mathbb{N}_{0}$ gives
\begin{align*}
\int \limits_{0}\limits^{t}\langle \Psi,\A^{\circ}\X(\tau,\omega)\rangle_{V}d\tau = - \sum \limits_{k =0}\limits^{m-1} \langle \Psi,\x_{k+1}(\omega)-\eta_{k+1}(\omega)-\x_{k}(\omega)\rangle_{V} - \langle \Psi,T_{\A}(t-\alpha_{m}(\omega))\x_{m}(\omega)-\x_{m}(\omega)\rangle_{V}.
\end{align*}
Now it is plain to deduce that also
\begin{align*}
\int \limits_{0}\limits^{t}\langle \Psi,\A^{\circ}\X(\tau,\omega)\rangle_{V}d\tau = \sum \limits_{k =0}\limits^{m-1} \langle \Psi,\eta_{k+1}(\omega)\rangle_{V} +\langle\Psi,x(\omega)\rangle_{V}-\langle\Psi,\X(t,\omega)\rangle_{V}.
\end{align*}
Finally, the previous equation yields
\begin{align*}
\langle\Psi,\X(t,\omega)-x\rangle_{V}+\int \limits_{0}\limits^{t}\langle\Psi,\A^{\circ}\X(\tau,\omega)\rangle_{V}d\tau = \sum \limits_{k =1}\limits^{m} \langle \Psi,\eta_{k}(\omega)\rangle_{V}= \sum \limits_{j =1}\limits^{\infty} \sum \limits_{k =1 }\limits^{j}\langle\Psi,\eta_{k}(\omega)\rangle_{V}\id_{[\alpha_{j}(\omega),\alpha_{j+1}(\omega))}(t),
\end{align*}
which gives (\ref{prop_exisproofeq2}). 
\end{proof}

\begin{lemma}\label{lemma_meassequ} Let $\mathcal{V}\subseteq V$ be a dense subspace of $V$. Then there is a sequence of mappings $(\Gamma_{n})_{n \in \mathbb{N}}$, with $\Gamma_{n}:V \rightarrow V$, such that the following assertions hold.
\begin{enumerate}
	\item $\Gamma_{n}(V)\subseteq \mathcal{V}$ for all $n \in \mathbb{N}$,
	\item $\Gamma_{n}$ is $\B(V)-\B(V)$-measurable for all $n \in \mathbb{N}$ and
	\item $\lim \limits_{n  \rightarrow \infty} \Gamma_{n}(v)=v$ for all $v \in V$.
\end{enumerate}
\end{lemma}
\begin{proof} As $\mathcal{V}$ is dense and $V$ is separable, we can find a sequence $(v_{n})_{n \in \mathbb{N}}\subseteq \mathcal{V}$ such that
\begin{align}
\label{lemma_meassequproofeq1}
\overline{\{v_{n}:n \in \mathbb{N}\}}=\overline{\mathcal{V}}=V.
\end{align}
Now introduce
\begin{align*}
V_{j,n}:=\{ v \in V: ~ ||v-v_{j}||_{V} = \min \limits_{k = 1,..,n } ||v-v_{k}||_{V} \},~\forall j \in \{1,...,n\} \text{ and } n \in \mathbb{N},
\end{align*}
 set $\tilde{V}_{1,n}:=V_{1,n}$ for all $n \in \mathbb{N}$ and
\begin{align*}
\tilde{V}_{j,n}:=V_{j,n} \setminus (V_{1,n} \cup .. \cup V_{j-1,n}),~\forall j\in \{2,..,n\} \text{ and } n \in \mathbb{N}\setminus \{1\}.
\end{align*}
Then it is plain that for each $n \in \mathbb{N}$ the system of sets $(\tilde{V}_{j,n})_{j=1,..,n}$ is a disjoint cover of $V$.\\
Now introduce $\Gamma_{n}:V\rightarrow V$ by
\begin{align*}
\Gamma_{n}(v):= \sum \limits_{j=1}\limits^{n} v_{j} \id_{\tilde{V}_{j,n}}(v),~\forall v \in V,,~n \in \mathbb{N}.
\end{align*}
Then it is plain that each $\Gamma_{n}$ only takes values in the set $\{v_{1},..,v_{n}\} \subseteq \mathcal{V}$ which gives i). In addition, we have that each $V_{j,n}$ is closed and therefore $V_{j,n}\in \B(V)$ which implies $\tilde{V}_{j,n}\in \B(V)$; this yields ii).\\
Finally, let us prove iii). To this end, fix $v \in V$ and note that for all $n \in \mathbb{N}$ there is precisely one $j(n) \in \{1,..,n\}$ such that $v \in \tilde{V}_{j(n),n}$ and hence $\Gamma_{n}(v)=v_{j(n)}$. Since also $v \in V_{j(n),n}$, we obtain
\begin{align*}
||v-\Gamma_{n}(v)||_{V}=||v-v_{j(n)}||_{V} = \min \limits_{k =1,..,n} ||v-v_{k}||_{V},~\forall n \in \mathbb{N}.
\end{align*}
Finally, (\ref{lemma_meassequproofeq1}) yields that there is for a given $\varepsilon>0$ an $n_{0} \in \mathbb{N}$ such that $||v-v_{n_{0}}||_{V}<\varepsilon$ and consequently 
\begin{align*}
||v-\Gamma_{n}(v)||_{V} = \min \limits_{k =1,..,n} ||v-v_{k}||_{V} < \varepsilon,~\forall n \geq n_{0},
\end{align*}
which concludes the proof.  
\end{proof}

\begin{theorem}\label{theorem_main} Let $\mathcal{V}\subseteq V$ be a dense subspace of $V$ and let $V^{\ast}\subseteq V^{\prime}$ be a subset which separates points. Moreover, let $x\in \mathcal{M}(\Omega;V)$, $\eta \in \mathcal{M}((0,\infty)\times Z \times \Omega;V)$ and let $\X$ denote the process generated by $(x,\eta)$. Finally, assume that $\mathcal{V}$ is an invariant set w.r.t. $T_{\A}$ and that $\langle \Psi,\A^{\circ}T_{\A}(\cdot)u\rangle_{V}\in L^{1}((0,t))$ for all $t>0$, $u \in \mathcal{V}$ and $\Psi \in V^{\ast}$.\\ Then the stochastic process $\X$ is the unique mild solution of (\ref{acprm})$\{x,\eta,V^{\ast}\}$. Moreover, if in addition $(0,0)\in \A$, we have
\begin{align}
\label{theorem_maineq1}
||\X(t)||_{V}\leq ||x||_{V}+\int \limits_{(0,t]\times Z} ||\eta(\tau,z)||_{V} N_{\Theta}(d\tau\otimes z),~\forall t \geq 0,
\end{align}
with probability one.
\end{theorem}
\begin{proof} Let $(\Gamma_{n})_{n \in \mathbb{N}}$, where $\Gamma_{n}:V\rightarrow V$, be such that $\Gamma_{n}(V)\subseteq \mathcal{V}$, $\Gamma_{n}$ is $\B(V)-\B(V)$-measurable and $\lim \limits_{n \rightarrow\infty}\Gamma_{n}(v)=v$ for all $v \in V$. In addition, let $M \in \F$ be a $\P$-null-set such that
\begin{align*}
0=\alpha_{0}(\omega)<\alpha_{1}(\omega)<\alpha_{2}(\omega)<...,~D(\Theta(\omega))=\{\alpha_{1}(\omega),\alpha_{2}(\omega),..\}\text{ and } \lim \limits_{m \rightarrow \infty } \alpha_{m}(\omega)=\infty,
\end{align*}
for all $\omega \in \Omega \setminus M$. Now introduce  $\eta_{k}(\omega):=\eta(\alpha_{k}(\omega),\Theta(\omega)(\alpha_{k}(\omega)),\omega)$ for all $\omega \in \Omega \setminus M$, $k \in \mathbb{N}$ and let $(\x_{m})_{m \in \mathbb{N}_{0}}$ be the sequence of jumps generated by $(x,\eta)$. Finally, for all $n \in \mathbb{N}$, let $(\x_{n,k})_{k \in \mathbb{N}_{0}}$ and $\X_{n}$ be the sequence and the process generated by $(\Gamma_{n}(x),\Gamma_{n}(\eta))$.\\ 
Firstly, note that $\Gamma_{n}(x) \in \mathcal{M}(\Omega;V)$ and $\Gamma_{n}(\eta) \in \mathcal{M}((0,\infty)\times Z \times \Omega;V)$, for all $n \in \mathbb{N}$, since the composition of measurable functions remains measurable. Moreover, it is plain that $\Gamma_{n}(x),~\Gamma_{n}(\eta)\in \mathcal{V}$ for all $n \in \mathbb{N}$ a.s. Consequently, we get by invoking Proposition \ref{prop_existence} that $\X_{n}$ is the strong solution of (\ref{acprm})$\{\Gamma_{n}(x),\Gamma_{n}(\eta),V^{\ast}\}$ for all $n \in \mathbb{N}$. Hence, it follows from Lemma \ref{lemma_solproperties} that $\X$ is a mild solution of (\ref{acprm})$\{x,\eta,V^{\ast}\}$, if
\begin{align}
\label{theorem_mainproof0}
\lim \limits_{n \rightarrow\infty}\int \limits_{(0,t]\times Z}||\Gamma_{n}(\eta(\tau,z,\omega))-\eta(\tau,z,\omega)||_{V}N_{\Theta}(d\tau\otimes z,\omega)=0,~\forall t>0 \text{ and } \omega \in \Omega \setminus M
\end{align}
and
\begin{align}
\label{theorem_mainproof1}
\lim \limits_{n \rightarrow\infty} \sup \limits_{\tau \in [0,t]} ||\X_{n}(\tau,\omega)-\X(\tau,\omega)||_{V}=0,~\forall t > 0 \text{ and } \omega \in \Omega \setminus M.
\end{align}
Now let $t>0$ and $\omega \in \Omega \setminus M$ be arbitrary but fixed and let $\tilde{m} \in \mathbb{N}_{0}$ be such that $t \in [\alpha_{\tilde{m}}(\omega),\alpha_{\tilde{m}+1}(\omega))$.\\
(\ref{theorem_mainproof0}) is trivial, since Lemma \ref{lemma_intex} gives that
\begin{align*}
\lim \limits_{n \rightarrow \infty} \int \limits_{(0,t]\times Z}||\Gamma_{n}(\eta(\tau,z,\omega))-\eta(\tau,z,\omega)||_{V}N_{\Theta}(d\tau\otimes z,\omega) = \lim \limits_{n \rightarrow \infty} \sum \limits_{k =1}\limits^{\tilde{m}}||\Gamma_{n}(\eta_{k}(\omega))-\eta_{k}(\omega)||_{V} = 0.
\end{align*}
Proof of (\ref{theorem_mainproof1}). Firstly, it will be proven inductively that
\begin{align}
\label{theorem_mainproof2}
||\x_{n,m}(\omega)-\x_{m}(\omega)||_{V} \leq || \Gamma_{n}(x(\omega))-x(\omega)||_{V}+\sum \limits_{k =1}\limits^{m}||\Gamma_{n}(\eta_{k}(\omega))-\eta_{k}(\omega)||_{V},~\forall m \in \mathbb{N}_{0}
\end{align}
and all $n \in \mathbb{N}$. If $m=0$, (\ref{theorem_mainproof2}) is trivial and if (\ref{theorem_mainproof2}) holds for an $m \in \mathbb{N}$, then applying Remark \ref{remark_ms}.i) and the induction hypothesis yields
\begin{eqnarray*}
||\x_{n,m+1}-\x_{m+1}||_{V}
& \leq & ~ ||\x_{n,m}(\omega)-\x_{m}(\omega)||_{V}+||\Gamma_{n}(\eta_{m+1}(\omega))-\eta_{m+1}(\omega)||_{V} \\
& \leq & ~ || \Gamma_{n}(x(\omega))-x(\omega)||_{V}+\sum \limits_{k =1}\limits^{m+1}||\Gamma_{n}(\eta_{k}(\omega))-\eta_{k}(\omega)||_{V},
\end{eqnarray*}
which proves (\ref{theorem_mainproof2}). Now note that for each $\tau \in [0,t]$ there is an $m_{\tau}\in \{0,...,\tilde{m}\}$, such that \linebreak$\tau \in [\alpha_{m_{\tau}}(\omega),\alpha_{m_{\tau}+1}(\omega))$. Consequently, appealing to Remark \ref{remark_ms}.i) and (\ref{theorem_mainproof2}) yields
\begin{eqnarray*}
||\X_{n}(\tau,\omega)-\X(\tau,\omega)||_{V} 
& = & ~ ||T_{\A}(\tau-\alpha_{m_{\tau}}(\omega))\x_{n,m_{\tau}}(\omega)-T_{\A}(\tau-\alpha_{m_{\tau}}(\omega))\x_{m_{\tau}}(\omega)||_{V} \\
& \leq & ~ ||\x_{n,m_{\tau}}(\omega)-\x_{m_{\tau}}(\omega)||_{V} \\
& \leq & ~ \max \limits_{m=0,..,\tilde{m}}||\x_{n,m}(\omega)-\x_{m}(\omega)||_{V} \\
& \leq & ~ || \Gamma_{n}(x(\omega))-x(\omega)||_{V}+\sum \limits_{k =1}\limits^{\tilde{m}}||\Gamma_{n}(\eta_{k}(\omega))-\eta_{k}(\omega)||_{V} 
\end{eqnarray*}
As this upper bound is independent of $\tau \in [0,t]$, we get
\begin{align*}
\lim \limits_{n \rightarrow \infty}\sup \limits_{ \tau \in [0,t]}||\X_{n}(\tau,\omega)-\X(\tau,\omega)||_{V} \leq \lim \limits_{n \rightarrow \infty}|| \Gamma_{n}(x(\omega))-x(\omega)||_{V}+\sum \limits_{k =1}\limits^{\tilde{m}}||\Gamma_{n}(\eta_{k}(\omega))-\eta_{k}(\omega)||_{V} =0,
\end{align*}
which proves (\ref{theorem_mainproof1}). Consequently, $\X$ is a mild solution of (\ref{acprm})$\{x,\eta,V^{\ast}\}$. Finally, Corollary \ref{corollary_unique} yields the uniqueness and Theorem \ref{theorem_bound}  gives (\ref{theorem_maineq1}).
\end{proof}

Finally, this section concludes by loosing some words on a particular choice of $\Theta$ and the drift $\eta$, which is considered in \cite{ich2}. The results proven in \cite{ich2} are solely based on the representation formula (\ref{def_soleq}). Introducing a process by (\ref{def_soleq}) actually requires fewer assumptions on $\A$ than in the current paper:

\begin{remark} Throughout this remark, just assume that $\A:D(\A)\rightarrow 2^{V}$ is densely defined and m-accretive\footnote{That means we drop the assumptions that $\A$ is domain invariant and admits an infinitesimal generator.}. Then the semigroup associated to $\A$ still exists (see Remark \ref{remark_ms}) and consequently the sequence as well as the process generated by $(x,\eta)$ are still well-defined, for any $x \in \mathcal{M}(\Omega;V)$ and $\eta \in \mathcal{M}((0,\infty)\times Z \times \Omega;V)$.\\
Now, fix $x \in \mathcal{M}(\Omega;V)$, introduce $(\eta_{k})_{k\in \mathbb{N}}\subseteq \mathcal{M}(\Omega;V)$ and define
\begin{align*}
\eta(t,z,\omega):= \sum \limits_{k =1}\limits^{\infty} \eta_{k}(\omega) \id_{[\alpha_{k}(\omega),\alpha_{k+1}(\omega))}(t),
\end{align*}
for all $t>0$, $z \in Z$ and $\omega \in \Omega$. Then it is obvious that indeed $\eta \in \mathcal{M}((0,\infty)\times Z \times \Omega;V)$.\\
Now assume in addition that there is an i.i.d. sequence of almost surely strictly positive random variables $(\beta_{m})_{m \in \mathbb{N}}$, such that the point process $\Theta$ fulfills $D(\Theta(\omega))=\{\beta_{1}(\omega),\beta_{1}(\omega)+\beta_{2}(\omega),...\}$ for $\P$-a.e. $\omega \in \Omega$. Then $\Theta$ is necessarily finite, since
\begin{align*}
\mathbb{E}N_{\Theta}((0,t]\times Z) = \mathbb{E} \sum \limits_{m=1}\limits^{\infty} \id\left\{\sum \limits_{k =1}\limits^{m}\beta_{k}\leq t\right\} = \sum \limits_{m=1}\limits^{\infty} \P\left(\sum \limits_{k =1}\limits^{m}\beta_{k}\leq t\right)<\infty,
\end{align*}
where the finiteness follows from \cite[Theorem 1.6]{renewal}.
Moreover, it is plain that in this case $\alpha_{m}=\sum \limits_{k =1}\limits^{m}\beta_{k}$ for all $m \in \mathbb{N}$.
In addition, in this case the sequence generated by $(x,\eta)$ fulfills $\x_{m}=T_{\A}(\beta_{m})\x_{m-1}+\eta_{m}$ for all $m \in \mathbb{N}$. Finally, it is easily seen that each $\x_{m}$ is $\F$-$\B(V)$-measurable and that the process generated by $(x,\eta)$ is still a stochastic process which has almost surely  c\`{a}dl\`{a}g paths. (These assertions follow solely from Remark \ref{remark_ms}.ii).)
\end{remark}

\section{Existence and Uniqueness for the weighted p-Laplacian evolution Equation}
\label{section_plaplace}
The purpose of this section is to demonstrate the applicability of the developed existence and uniqueness results to the weighted p-Laplacian evolution equation with Neumann boundary conditions on an $L^{1}$-space.\\
The reader is referred to \cite{mazon}, for existence and uniqueness results of the deterministic weighted $p$-Laplacian evolution equation. Moreover, \cite{acmbook} and \cite{mainbasedon} contain many other examples of nonlinear evolution equations. Finally, \cite{cao} contains a useful criteria regarding domain invariance and differentiability almost everywhere of nonlinear semigroups, which is probably not just in our example useful to prove the needed assumptions on $\A$.\\

Throughout this section let $n \in \mathbb{N}\setminus \{1\}$  and $\emptyset \neq  S  \subseteq \mathbb{R}^{n}$ be a non-empty, open, connected and bounded sets of class $C^{1}$. Additionally, for any $q \in [1,\infty]$ and $m \in \mathbb{N}$, we set \linebreak$L^{q}(S;\mathbb{R}^{m}):=L^{q}(S,\B(S),\lambda;\mathbb{R}^{m})$, where $\lambda$ is the Lebesgue measure. Moreover, this is further abbreviated by $L^{q}(S)$, if $m=1$.\\
Now, let $p \in (1,\infty)\setminus \{2\}$, introduce $\gamma: S   \rightarrow (0,\infty)$ and assume that $\gamma \in L^{\infty}( S )$, $\gamma^{\frac{1}{1-p}} \in L^{1}( S)$ and that there is an $A_{p}$-Muckenhoupt weight (see, \cite[page 4]{ich1}) $\gamma_{0}:\mathbb{R}^{n}\rightarrow \mathbb{R}$ such that $\gamma_{0}|_{ S }=\gamma$ a.e.  on $S$.
Moreover, $W_{\gamma}^{1,p}( S )$ denotes the weighted Sobolev space defined by
\begin{align*}
W_{\gamma }^{1,p}( S ):=\{f \in L^{p}( S ): f \text{ is weakly diff. and } ~\gamma^{\frac{1}{p}}\nabla f \in L^{p}( S;\mathbb{R}^{n})\}. 
\end{align*} 
Throughout this section, $|\cdot|_{n}$ denotes the Euclidean norm on $\mathbb{R}^{n}$ and for any $x,y\in\R^{n}$, $x\cdot y$ is the canonical inner product of these vectors.\\
Using these notations, we introduce the following weighted p-Laplacian operator with Neumann boundary conditions:

\begin{definition}\label{definifition_plaplaceop} Let $A: D(A)\rightarrow 2^{L^{1}(S)}$ be defined by: $(f,\hat{f}) \in A$ if and only if the following assertions hold.
	\begin{enumerate}
		\item $f \in W^{1,p}_{\gamma}( S ) \cap L^{\infty}( S )$. 
		\item $\hat{f} \in L^{1}( S )$.
		\item $\int \limits_{ S}  \gamma|\nabla f|_{n}^{p-2}\nabla f\cdot\nabla \varphi  d \lambda = \int \limits_{ S } \hat{f} \varphi d \lambda$ for all $\varphi\in W^{1,p}_{\gamma }( S )\cap L^{\infty}( S )$.
	\end{enumerate}
\end{definition}

\begin{remark} It is an easy exercise to see that the integrals occurring in Definition \ref{definifition_plaplaceop}.iii) exist and are finite. Moreover, one also easily verifies that $A$ is single-valued, see \cite[Lemma 3.1]{ich1}.\\
In addition,  note that if one would choose $\gamma=1$ on $S$, then $A$ is simply the $p$-Laplacian operator with Neumann boundary conditions.
\end{remark}

\begin{remark} It turns out that $A$ is not m-accretive but that its closure is. In the sequel\linebreak $\A:D(\A)\rightarrow 2^{L^{1}(S)}$ denotes the closure of $A$, i.e. $(f,\hat{f})\in \A$ if there is a sequence $((f_{m},\hat{f}_{m}))_{m \in \mathbb{N}}\subseteq A$ such that $\lim \limits_{m \rightarrow \infty} (f_{m},\hat{f}_{m})=(f,\hat{f})$,in $L^{1}(S)\times L^{1}(S)$.\\
Actually, it is possible to determine the closure explicitly, see \cite[Proposition 3.6]{mazon} or \cite[Definition 2.2]{ich1}. But the explicit description of the closure is quite technical and not needed for our purposes, therefore it will be omitted.
\end{remark}

Now introduce $J_{0}$ as the space of all convex, lower semi-continuous functions $j:\mathbb{R}\rightarrow [0,\infty]$ fulfilling $j(0)=0$. Given $f,h \in L^{1}( S )$, we write $f<<h$ whenever
\begin{align*}
\int \limits_{ S }j \circ f d\lambda \leq \int \limits_{ S }  j \circ h d \lambda,~\forall j \in J_{0}.
\end{align*}  
Moreover, an operator $B:D(B)\rightarrow 2^{L^{1}(S)}$ is called completely accretive, if $f-h<<f-h+\alpha (\hat{f}-\hat{h})$ for all $(f,\hat{f}),~(h,\hat{h})\in B$ and $\alpha \in (0,\infty)$. The reader is referred to \cite{cao} for a detailed discussion of the concept of complete accretivity.

\begin{remark}\label{remark_linfest} $\A$ is densely defined, m-accretive and completely accretive, see \cite[Theorem 3.7]{mazon}. In the sequel, $T_{\A}(\cdot)u:[0,\infty)\rightarrow L^{1}(S)$ denotes the semigroup associated to $\A$, see Remark \ref{remark_ms}.\\
Moreover, we have $T_{\A}(t)u<<u$ for all $t \geq 0$ and $u \in L^{1}(S)$, see  \cite[Lemma 3.3]{ich1}. Consequently, it is an easy exercise, to verify that
\begin{align}
\label{remark_linfesteq}
||T_{\A}(t)u||_{L^{q}(S)} \leq ||u||_{L^{q}(S)}
\end{align}
for all $t\geq 0$, $u \in L^{q}(S)$ and $q\in [1,\infty]$.
\end{remark}

\begin{lemma}\label{lemma_rela} $T_{\A}$ admits an infinitesimal generator $\A^{\circ}:L^{1}(S)\rightarrow L^{1}(S)$. In addition, we have $\A^{\circ}v=Av$ for all $v \in D(\A)\cap L^{\infty}(S)=D(A)$.
\end{lemma}
\begin{proof} Let $\A^{\circ}:L^{1}(S) \rightarrow L^{1}(S)$ be defined by $\A^{\circ}v:=0$ if $v \not \in D(\A)$ and if $v \in D(\A)$, set $\A^{\circ}v:=\hat{v}$, where $\hat{v} \in \A v$ is the uniquely determined element fulfilling $\hat{v}<<\tilde{v}$ for all $\tilde{v} \in \A v$. (\cite[Proposition 3.7.iii]{cao} gives that there is indeed exactly one element fulfilling this.)\\
Now appealing to \cite[Theorem 4.2]{cao} yields
\begin{align*}
\lim \limits_{h \searrow 0} \frac{T_{\A}(h)v-v}{h} =-\A^{\circ}v,~\forall v \in D(\A),
\end{align*}
which proves that $T_{\A}$ admits an infinitesimal generator.\\
Moreover, \cite[Lemma 3.1]{ich1} gives that if $v \in D(\A)\cap L^{\infty}(S)$ and $\hat{v} \in \A v$, then $v \in D(A)$ and $\hat{v} = Av$. Consequently, we have $D(\A)\cap L^{\infty}(S) \subseteq D(A)$. Now note that it is plain that $D(A)\subseteq L^{\infty}(S)$ and $D(A)\subseteq D(\A)$, which yields $D(A)=D(\A)\cap L^{\infty}(S)$.\\
The preceding observation also yields that if $v \in D(\A)\cap L^{\infty}(S)$, the set $\A v$ contains only one element, which is $Av$. As also $ \A^{\circ}v \in \A v$, we have $\A^{\circ}v=Av$.
\end{proof}

\begin{lemma}\label{lemma_rege} $T_{\A}$ is domain invariant. In addition, we have
\begin{align*}
|\A^{\circ} T_{\A}(t)v| \leq 2 \frac{|v|}{|p-2|t}
\end{align*}
a.e. on $S$, for all $t>0$ and $v \in L^{1}(S)$.
\end{lemma}
\begin{proof} As $\A$ is densely defined, m-accretive and completely accretive, \cite[Theorem 4.4]{cao} yields that it suffices to prove that $\A$ is positively homogeneous of degree $p-1$; which is true, see \cite[Theorem 3.7]{mazon}.
\end{proof}

The following lemma enables us to apply Proposition \ref{prop_existence} and Theorem \ref{theorem_main} to the (closure of the) weighted $p$-Laplacian evolution equation. As the reader probably guessed correctly, the Banach Space $V$ considered in Section \ref{section_exun} has to be chosen as $V=L^{1}(S)$. As usually, we identify $V^{\prime}$ with $L^{\infty}(S)$. Note that in this case, the duality $\langle \cdot, \cdot \rangle_{L^{1}(S)}$ reduces to an integral, i.e.
\begin{align*}
\langle f ,h \rangle_{L^{1}(S)} = \int \limits_{S} f h d\lambda,
\end{align*}
for any $f \in L^{1}(S)$ and $h \in L^{\infty}(S)$.

\begin{proposition} We have 
\begin{align}
\label{theorem_mainplaplaceexunproofeq0}
\langle \Psi, \A^{\circ}T_{\A}(\cdot)v\rangle_{L^{1}(S)} \in L^{1}(0,t),
\end{align}
for all $t>0, \Psi \in W^{1,p}_{\gamma}(S)\cap L^{\infty}(S)\text{ and } v \in L^{\infty}(S)$.
\end{proposition}
\begin{proof} Firstly, note that $T_{\A}(\tau)v \in L^{\infty}(S) \cap D(\A)$, for all $\tau >0$, since $v \in L^{\infty}(S)$. Consequently, Lemma \ref{lemma_rela} gives $\A^{\circ}T_{\A}(\tau)v=AT_{\A}(\tau)v$, for all $\tau>0$. Using this as well as Cauchy-Schwarz' and H\"older's inequality yields
\begin{eqnarray*}
|\langle \Psi, \A^{\circ}T_{\A}(\tau)v\rangle_{V}| 
& = & ~ \left|\int \limits_{S}\gamma  |\nabla T_{\A}(\tau)v|_{n}^{p-2}\nabla T_{\A}(\tau)v\cdot \nabla \Psi d\lambda\right| \\
& \leq & ~ \left(\int \limits_{S}\gamma |\nabla T_{\A}(\tau)v|_{n}^{p}d\lambda \right)^{\frac{p-1}{p}}\left( \int \limits_{S}\gamma|\nabla \Psi|_{n}^{p}d\lambda\right)^{\frac{1}{p}}
\end{eqnarray*}
Moreover, appealing to Lemma \ref{lemma_rela} again yields $T_{\A}(\tau)v \in D(A)\subseteq W^{1,p}_{\gamma}(S)\cap L^{\infty}(S)$ for all $\tau>0$. Consequently, we get
\begin{align*}
|\langle \Psi, \A^{\circ}T_{\A}(\tau)v\rangle_{V}| \leq \left(\int \limits_{S} T_{\A}(\tau)v AT_{\A}(\tau)vd\lambda \right)^{\frac{p-1}{p}}\left( \int \limits_{S}\gamma|\nabla \Psi|_{n}^{p}d\lambda\right)^{\frac{1}{p}}
\end{align*}
Hence we have by using this, together with Lemma \ref{lemma_rege} and (\ref{remark_linfesteq}) that
\begin{align*}
|\langle \Psi, \A^{\circ}T_{\A}(\tau)v\rangle_{V}| \leq \left(\frac{1}{\tau}2\lambda(S)\frac{1}{|p-2|}||v||_{L^{\infty}(S)}^{2}\right)^{\frac{p-1}{p}}\left( \int \limits_{S}\gamma|\nabla \Psi|_{n}^{p}d\lambda\right)^{\frac{1}{p}},\forall \tau>0.
\end{align*}
The preceding inequality obviously implies (\ref{theorem_mainplaplaceexunproofeq0}).
\end{proof}

In the sequel, we denote for any $x\in \mathcal{M}(\Omega;L^{1}(S))$ and $\eta \in \mathcal{M}((0,\infty)\times Z \times \Omega;L^{1}(S))$, by \linebreak$\X_{x,\eta}:[0,\infty)\times \Omega \rightarrow L^{1}(S)$ the process generated by $(x,\eta)$.

\begin{theorem} Let $x\in \mathcal{M}(\Omega;L^{1}(S))$ and $\eta \in \mathcal{M}((0,\infty)\times Z \times \Omega;L^{1}(S))$. Then $\X_{x,\eta}$ is the uniquely determined mild solution of (\ref{acprm})$\{x,\eta,W^{1,p}_{\gamma}(S)\cap L^{\infty}(S)\}$.\\ 
If in addition $x \in L^{\infty}(S)$  and $\eta(t,z)\in L^{\infty}(S)$ for all $t>0$ and $z \in Z$ a.s., then $\X_{x,\eta}$ is even the uniquely determined strong solution of (\ref{acprm})$\{x,\eta,W^{1,p}_{\gamma}(S)\cap L^{\infty}(S)\}$.	
\end{theorem}
\begin{proof} The claim follows from Proposition \ref{prop_existence} and Theorem \ref{theorem_main}, by choosing $V=L^{1}(S)$, $\mathcal{V}=L^{\infty}(S)$ and $V^{\ast}=W^{1,p}_{\gamma}(S)\cap L^{\infty}(S)$ there. Note that $\mathcal{V}$ is indeed dense in $V$ and invariant w.r.t. $T_{\A}$ and that $V^{\ast}$ separates points. (The denseness is common knowledge, the invariance follows from (\ref{remark_linfesteq}) and the point-separation follows for example from the fact that $C_{c}^{\infty}(S)\subseteq W^{1,p}_{\gamma}(S)\cap L^{\infty}(S)$.)
\end{proof}

\begin{theorem}\label{theorem_plaplacebound} Let $x\in \mathcal{M}(\Omega;L^{1}(S))$ and $\eta \in \mathcal{M}((0,\infty)\times Z \times \Omega;L^{1}(S))$. Then we have
\begin{align*}
||\X_{x,\eta}(t)||_{L^{1}(S)}\leq ||x||_{L^{1}(S)}+\int \limits_{(0,t]\times Z} ||\eta(\tau,z)||_{L^{1}(S)} N_{\Theta}(d\tau\otimes z),~\forall t \geq 0,
\end{align*}
almost surely. Moreover $(x,\eta)$ can be chosen such that
\begin{align*}
||\X_{x,\eta}(t)||_{L^{1}(S)} = ||x||_{L^{1}(S)}+\int \limits_{(0,t]\times Z} ||\eta(\tau,z)||_{L^{1}(S)} N_{\Theta}(d\tau\otimes z),~\forall t \geq 0,
\end{align*}
with probability one.
\end{theorem}
\begin{proof} The former part of the claim follows directly from Theorem \ref{theorem_main}, which is applicable, since one instantly verifies that $(0,0)\in \A$.\\
Now let us prove the latter. To this end,  it will be useful that $T_{\A}(t)\varphi = \varphi$ for all $t\geq 0$ and $\varphi: S \rightarrow \mathbb{R}$ which are constant, cf. \cite[Lemma 4.1]{ich1}.\\
Now, let $\varphi_{1},~\varphi_{2}:\Omega\rightarrow L^{1}(S)$ be such that both of them are for $\P$-a.e. $\omega \in \Omega$ constant, nonnegative functions on $S$ and chose $x:=\varphi_{1}$ and $\eta(t,z,\cdot)=\varphi_{2}$ for all $t>0,~z\in Z$. In addition, let $(\x_{m})_{m \in \mathbb{N}}$ be the sequence generated by $(x,\eta)$. Then a simple induction yields
\begin{align*}
\x_{m}=\varphi_{1}+m\varphi_{2},~\forall m \in \mathbb{N}_{0},
\end{align*}
almost surely. Consequently, we have
\begin{align*}
\X_{x,\eta}(t)= \sum \limits_{m = 0}\limits^{\infty}T_{\A}((t-\alpha_{m})_{+})\x_{m}\id_{[\alpha_{m},\alpha_{m+1})}(t)=\sum \limits_{m = 0}\limits^{\infty}(\varphi_{1}+m\varphi_{2})\id_{[\alpha_{m},\alpha_{m+1})}(t),
\end{align*}
for all $t \geq 0$ almost surely. Moreover, we get
\begin{align*}
||\X_{x,\eta}(t)||_{L^{1}(S)}=\lambda(S)\sum \limits_{m = 0}\limits^{\infty}(\varphi_{1}+m\varphi_{2})\id_{[\alpha_{m},\alpha_{m+1})}(t).
\end{align*}
In addition, it is plain that $||x||_{L^{1}(S)}=\lambda(S)\varphi_{1}$. Now, appealing to Lemma \ref{lemma_intex} yields
\begin{align*}
\int \limits_{(0,t]\times Z} ||\eta(\tau,z)||_{L^{1}(S)} N_{\Theta}(d\tau\otimes z) = \sum \limits_{m =1}\limits^{\infty} \sum \limits_{k =1 }\limits^{m}\lambda(S)\varphi_{2}\id_{[\alpha_{m},\alpha_{m+1})}(t) = \sum \limits_{m =1}\limits^{\infty} m\lambda(S)\varphi_{2}\id_{[\alpha_{m},\alpha_{m+1})}(t),
\end{align*}
for all $t\geq 0$ almost surely. Finally, putting it all together gives
\begin{align*}
	||\X_{x,\eta}(t)||_{L^{1}(S)} = ||x||_{L^{1}(S)}+\int \limits_{(0,t]\times Z} ||\eta(\tau,z)||_{L^{1}(S)}, N_{\Theta}(d\tau\otimes z)
\end{align*}
for all $t \geq 0$, with probability one.
\end{proof}

\newpage 



\begin{thebibliography}{9} 
\addcontentsline{toc}{chapter}{\ \quad Bibliography}

\bibitem{acmbook}
{\sc F. Andreu-Vaillo, V. Caselles, J.M. Maz\'{o}n}, {\em Parabolic Quasilinear Equations minimizing linear growth Functionals, Birkh\"auser, 2010\/}

\bibitem{mazon}
{\sc F. Andreu, J.M. Maz\'{o}n,  J. Rossi, J. Toledo}, {\em Local and nonlocal weighted p-Laplacian evolution Equations with Neumann Boundary Conditions, Publ. Math. 55 (2011)  27-66\/}

\bibitem{mainbasedon}
{\sc F. Andreu , J.M. Maz\'{o}n, S. Segura de Le\'{o}n, J. Toledo}, {\em Quasi-linear elliptic and parabolic Equations in $L^{1}$ with nonlinear boundary Conditions, Adv. in Math. Sci. and Appl. (1997) 183-213\/} 

\bibitem{greenbook}
{\sc W. Arendt, C. Batty, M. Hieber, F. Neubrander}, {\em Vector-valued Laplace Transforms and Cauchy Problems, Birkh\"auser, 2010\/}

\bibitem{cao}
{\sc P.  B\'{e}nilan, M. Crandall}, {\em Completely Accretive Operators, In: P. Clement, E. Mitidieri, B. de Pagter, (eds.) Semigroup Theory and Evolution Equation. Marcel Dekker Inc., New York, pp. 41-76 (1991)}

\bibitem{BenilanBook}
{\sc P. B\'{e}nilan, M. Crandall, A. Pazy}, {\em Nonlinear Evolution Equations in Banach Spaces, Book to appear\/}; \url{http://www.math.tu-dresden.de/~chill/files/} 

\bibitem{Billingsley}
{\sc P. Billingsley}, {\em Convergence of Probability Measures, Wiley, 1999\/}

\bibitem{E}
{\sc G. Edgar}, {\em Measure, Topology and Fractal Geometry, Springer, 2008\/}

\bibitem{kechris}
{\sc A. Kechris}, {\em Classical Descriptive Set Theory, Springer, 1995\/} 

\bibitem{knoche}
{\sc C. Knoche}, {\em Mild Solutions of SPDE’s Driven by Poisson Noise in Infinite Dimensions and their Dependence on Initial Conditions, 2005\/} (Dissertation)

\bibitem{renewal}
{\sc K. Mitov, E. Omey}, {\em Renewal Processes, Springer, 2014\/} 

\bibitem{SIBS}
{\sc V. Mandrekar, B. R\" udiger}, {\em Stochastic Integration in Banach Spaces, Springer, 2015\/} 

\bibitem{ich1}
{\sc A. Nerlich}, {\em Asymptotic Results for Solutions of a weighted $p$-Laplacian evolution Equation with Neumann Boundary Conditions, Nonlinear Differ. Equ. Appl.  (2017) 24-46 \/}

\bibitem{ich2}
{\sc A. Nerlich}, {\em Abstract Cauchy Problems driven by random Measures: Asymptotic Results in the finite extinction Case, (Submitted)\/} 

\bibitem{pointsep}
{\sc N. Vakhania, V. Tarieladze}, {\em Probability Distributions on Banach Spaces, D. Reidel Publishing Company, 1987\/} 

\end{thebibliography}
\end{document}